\documentclass{amsart}
 \usepackage{amsmath}
 \usepackage{epsfig} 
 \usepackage{graphics} 

\newtheorem{theorem}{Theorem}
\newtheorem{corollary}{Corollary}

\newtheorem{lemma}{Lemma}
\newtheorem{proposition}{Proposition}

\newtheorem*{problem}{Problem}
\newtheorem{problem 1}{Problem 1}
\newtheorem{problem 2}{Problem 2}
\newtheorem{problem 3}{Problem 3}
\theoremstyle{definition}

\newtheorem{definition}{Definition}
\newtheorem{remark}{Remark}

\newtheorem{example}{Example}

\newcommand{\bea}{\begin{eqnarray*}}
\newcommand{\eea}{\end{eqnarray*}}

\title[LAMINATIONS WITH SINGULARITIES ]
      {RIEMANN SURFACE LAMINATIONS WITH SINGULARITIES}

\author[John Erik Forn\ae ss, Nessim Sibony]{}

\email{fornaess@umich.edu, nessim.sibony@math.u-psud.fr}

\thanks{Preliminary Version, June 2007. }

\begin{document}

\begin{abstract}
In these introductory notes we give the basics of the theory of holomorphic foliations and laminations. The emphasis is on the theory of harmonic currents and unique ergodicity for laminations transversally Lipschitz in $\mathbb P^2$ and for generic holomorphic foliations in $\mathbb P^2.$
  \end{abstract}

\maketitle



\centerline{John Erik Forn\ae ss\footnote{The first author is 
supported by an NSF grant.
} and Nessim Sibony
}

\medskip


\medskip

 \medskip

\tableofcontents

\section{Introduction}

In applications of mathematics to sciences, differential equations
play an important role. An important such equation is
\bea
x_1'(t) & = & F_1(x)\\
& \cdots &\\
x_k'(t) & = & F_k(x)\\
\eea

The solutions are integral curves of the vector field

$$
Z(x) = \sum_j F_j(x) \frac{\partial}{\partial x_j}
$$

The integral curves are tangent to the common null space of
the $1-$ forms $\lambda_{ij}=F_i(x) dx_j-F_j(x)dx_i.$
We say that the integral curves are weakly directed by the $\lambda_{i,j}.$

To understand the global nature of solutions it is useful to investigate
their behavior near singular points, which are the common zeros of the $F_i$,
i.e. the fixed points of the flow. These common zeros are frequently complex. This leads us to study laminations with singularities by Riemann surfaces in $\mathbb C^k.$ The Riemann surfaces are the integral curves in
complex time of the complexification of $Z.$ We consider the case when the 
 $F_i$ are polynomials and working in the complex we can just as well assume that the $F_i(z)$ are holomorphic polynomials. (These of course can then be considered as real vector fields on $\mathbb R^{2k}.$) Moreover it is helpful in order to understand the global behavior to consider the extension to $\mathbb P^k$. In fact then it is natural to study general holomorphic
foliations by Riemann surfaces with singularities of $\mathbb P^k$
which are weakly directed by complex one forms
$\lambda_{\ell}.$
Petrovski and Landis \cite{P}, who studied Hilbert's problem on the number of cycles of a polynomial vector field in the plane, had the idea to relate the real cycles to algebraic solutions for the extension of
 the vector field to $\mathbb C^2$. Indeed there was a gap in their proof, but their geometric method raised interesting problems for holomorphic foliations. It also raises the general question of proving
 results for real vector fields, analogous to the complex case, for example such as results described in this survey.
 Our motivation is to study global properties of holomorphic foliations with singularities.
In particular we want to introduce objects describing the asymptotic behavior of solutions
of a differential equation in $\mathbb C^2$, $\frac{dw}{dz} = \frac{Q(z,w)}{P(z,w)},$ where $P$ and
$Q$ are polynomials.

\medskip

Let us recall the more familiar case of discrete dynamical systems.
In order to study the global dynamical behavior of a continuous map $f:X \rightarrow X$ on a compact
space $X,$ it is useful to introduce invariant probability measures $\mu$, satisfying
$f_*\mu=\mu.$ They can be obtained following Krylov-Bogoloubov as limit points of
$\frac{1}{N} \sum_{i=0}^{N-1}\delta_{f^i(x)},$ where $\delta_{f^i(x)}$ denotes the Dirac mass at the
point $f^i(x).$ It is then of interest to study the ergodic properties of such invariant measures, in particular to show that the limit of averages along an orbit is independent of the orbit.
\medskip

 We want to construct an analogue for laminations with singularities by Riemann surfaces and study their ergodic
properties. If $x(t):\mathbb R \rightarrow \mathbb R^k$ parametrizes
an integral curve of $Z$, we can define a measure $\mu_x$ 
on the curve by
letting the image $x([a,b])$ have mass $b-a.$ Then limit points of the measures
$\frac{\mu_{|{x(0,N)}}}{\mu(x(0,N))}$ are invariant under the flow. In fact we will discuss in Section 5.8  a natural invariant measure which is invariant under a gradient flow along leaves. One can also introduce invariant $1-$ currents which are weak limits of $\frac{|x'(t)|^{-1}[x(0,N)]}{\mu(x(0,N))}$
where $[x(0,N)]$ denotes the current of integration. One might ask whether such limits are closed or not and whether they have a laminar structure as well as ergodic properties. A disadvantage in using time as a parameter is that the vector fields are usually not complete. We are interested in questions like density of leaves, so we might prefer other parametrizations. In the complex case we can use universal covering maps of leaves to get analogous measures and currents.
The currents we get will be closed or harmonic.

\medskip
There are several surveys on holomorphic foliations and laminations,
Brunella \cite{Br2000}, Ghys \cite{Gh1999}. Our purpose here is to give an introduction to harmonic currents on singular foliations as developped  by the authors, harmonic currents is the analogue of invariant measures in discrete dynamics.

 For foliations in $\mathbb P^2$ with only hyperbolic singularities (see section 2.2 for a precise definition) and without algebraic leaves, we will introduce an averaging process along leaves and we will show that the limit is unique, i.e. independent of the leaf, a unique ergodicity  property, see  \cite{FS2005}, \cite{FS2006}. 
We introduce the basic concepts and results on foliations and laminations with emphasis on singularities. We also give proofs of some results in the simplest interesting case. The problem of unique ergodicity on minimal sets in higher dimension, for foliations by Riemann surfaces is open.

\section{Basic results}

\subsection{Definitions and examples in dimension $2$}

\begin{definition}
Let $Y$ be a Hausdorff topological  space. Then $(Y,\mathcal L)$ is a lamination by Riemann surfaces if
${\mathcal L}$ is an atlas with charts
$$
\phi:U \rightarrow B \times T_U
$$
\noindent where $B$ is the unit disc in $\mathbb C$, $T_U$ is a topological space and $\phi$
is a homeomorphism. The change of coordinates are of the form
$$
(z,t) \rightarrow (z',t'),\; {\mbox{with}}\; z'=h(z,t),t'=t'(t)
$$
\noindent where $h$ is continuous and holomorphic with respect to $z.$
\end{definition}

Observe that this implies in particular that the functions $\frac{\partial^\alpha h}{\partial z^\alpha}$ are
continuous with respect to both variables.

\medskip

Most of the time we will consider the following situation. The set $Y$ is laminated and locally closed in a
compact complex manifold $M.$ Let $X:=\overline{Y}$ and $E:=X \setminus Y.$ Therefore $E$ is closed
and will be considered as a singularity set. Then $X\setminus E$ has the structure of a
lamination by Riemann surfaces and we will consider $(X,\mathcal L,E)$ as a lamination with
singularities.

\medskip

Let $\phi=(h,t)$ be a chart defined on $U\subset Y.$ Then $U$ is said to be a flow box. A connected component 
$(t=c)$ in $U$ is called a plaque. A leaf $L$ is a minimal connected subset of $X \setminus E$
such that if $L$
intersects a plaque then $L$ contains the plaque. A transversal in a flow box is a closed
set which intersects every plaque in one point.

\medskip

If $L$ is a leaf in $(X,\mathcal L,E)$, then $\overline{L}$ has also the structure of a lamination
with singularities. Indeed if $q\in \overline{L}\setminus E,$ the leaf through $q$ is also in $\overline{L},$ this follows from the local structure in a flow box.

\medskip

The definition can be extended to get the notion of a lamination of class ${\mathcal C}^k$. In that
case $T $ is not an arbitrary topological space but a closed set in a Euclidean space and the functions $h,t'$ are assumed to be of class ${\mathcal C}^k.$

We first give some examples and definitions in complex dimension $2$.

\begin{example}
Let $\alpha=P(z,w)dw-Q(z,w)dz$ be a holomorphic $1-$ form in $\mathbb C^2$, with
$P,Q$ relatively prime polynomials. Denote by ${\mathcal F}_\alpha$ the foliation tangent at each point to the kernel of $\alpha.$ Equivalently the foliation is tangent to  the vector field $Z_\alpha=P\frac{\partial}{\partial z}+Q\frac{\partial}{\partial w}$
in $\mathbb C^2 \setminus \{P=Q=0\}$ i.e. $\mathbb C^2$ with a finite number of points deleted.
At a point $p$ where $P$ and $Q$ do not vanish simultaneously the "plaque" through $p$
is given by $\zeta \rightarrow (z(\zeta),w(\zeta))$ with $(z,w)$ holomorphic functions defined on a disc
$D(0,r)$ and satisfying the differential equation
$$
P(z(\zeta),w(\zeta))w'(\zeta)-Q(z(\zeta),w(\zeta))z'(\zeta)=0
$$
\noindent $\zeta\in D(0,r).$ Since in any chart $\mathcal F_\alpha={\mathcal F}_{f \alpha}$ where
$f$ is a holomorphic nonvanishing function, it is possible to extend the foliation to $\mathbb P^2.$
However if we express $\alpha$ in another chart as a form with polynomial coefficients, $\max({\mbox{deg}}(P), {\mbox{deg}}(Q))$ is not invariant. 
\end{example}

The degree of $\mathcal F$ is defined to be the number of tangencies of $\mathcal F$ with a generic line in $\mathbb P^2.$ When $\mathcal F$ is of degree $d,$ then in an affine chart, the vector field $Z$ is of the form
$$
Z=[P(z,w)+zR(z,w)]\frac{\partial }{\partial z}+[Q(z,w)+wR(z,w)]\frac{\partial}{\partial w}
$$
\noindent with $P,Q$ polynomials of degree $\leq d$ and $R$ is a homogeneous polynomial of degree $d$, such that if $R$ is identically zero, then $\max($deg($P),$ deg$(Q))=d$ and $P,Q$ are not of the form $P=zH,Q=wH.$
It can be shown that a holomorphic foliation on a complex   surface, is given locally by the integral curves of holomorphic vector field, with isolated singularities. Locally, the vector field is unique up to multiplication by a non vanishing holomorphic function.

\medskip
We describe the isolated singularities in dimension $2.$ A theorem of Seidenberg \cite{S1968} states that any isolated singularity for a germ of a holomorphic foliation in dimension $2$, can be transformed after finitely many blow ups to a reduced singularity as described in the following definition. Assume the singular point is $0$, and that the spectrum of the $(2,2)$ matrix $DZ_\alpha(0)$ is given by the two complex numbers $(\lambda_1,\lambda_2).$

\begin{definition}
The germ $\alpha$ is $\underline{\mbox{reduced}}$ if one of the following holds:\\
\noindent a) $\lambda_1 \lambda_2 \neq 0$ and $\lambda:=\lambda_1/\lambda_2\notin 
{\mathbb N}\cup 1/{\mathbb N},$\\
\noindent b) $\lambda_1 \neq 0, \lambda_2=0 $, in which  case, the singularity  is called a saddle node.
\end{definition}

The germ is in the Poincar\'e domain if $\lambda \notin \mathbb R^-\cup\{0\},$ i.e.
the segment in $\mathbb C$ connecting $\lambda_1,\lambda_2$ does not pass through the origin.
When $\lambda \in \mathbb R^-\setminus \{0\}$ the singularity is in the Siegel domain.
If $\lambda \in \mathbb Q^-,$ then we have resonances, as defined in full generality in the next paragraph.

When the singularity is reduced and $\lambda \notin \mathbb R^-\cup\{0\}$, then a theorem of Poincar\'e asserts that there is a holomorphic coordinate system where $\alpha$
can be read as $\alpha=zdw-\lambda wdz.$ 
The solutions of $\alpha=0$ can be given explicitly by 
 $$z(\zeta)=e^\zeta, w(\zeta) =c e^{\lambda \zeta}, c \neq 0, (z=0), (w=0).$$
Recall that a $\underline{\mbox{separatrix}}$ at a singularity $p,$ is a leaf $L$ in a neighborhhood $V$ of $p$, such that $L\cup\{p\}$ is an irreducible 
holomorphic curve.Observe that  $L\cup\{p\}$ could be singular.

\begin{theorem} (Camacho-Sad)  \cite{CS1982} 
Every germ of holomorphic vector field in $(\mathbb C^2,0)$ with isolated singularity at $0$ admits a separatrix.
\end{theorem}

The result is false for holomorphic vector fields in $(\mathbb C^3,0)$, see \cite{GML1992}.

In the above example, the solution curves $z=0,w=0$ are the only separatrices if $\lambda \notin \mathbb Q.$ When $\lambda \notin \mathbb R,$ i.e. $\Im(\lambda) \neq 0,$ the singularity is said to be hyperbolic. It is easy to check that the solutions near the singularity cluster
along the separatrices. 
When $\lambda \in \mathbb R^-$, to have linearization it is in general necessary to assume
that $\lambda$ is poorly approximable by rationals (Siegel-Brjuno condition) 
\cite{AI1988}.  When $\lambda \in \mathbb Q^-,$ the singularity is either linearizable or is of flower type,  [9].

\begin{example}
A very special case, with a non hyperbolic singularity is given by the form $\alpha=zdw-\lambda wdz,
\lambda \in \mathbb R \setminus \mathbb Q.$ Then $X(c) =\{|w|=c|z|^\lambda\}, c>0$ is an example
of a lamination with singularities. When we take the closure of $X(c)$ in $\mathbb P^2$ it's
smooth except at two of the points $[1:0:0], [0:0:1]$ and $[0:1:0].$ Observe that since $\lambda$ is irrational
every leaf is dense in $X(c).$ An $X(c)$ does not pass through all singular points, but any two will intersect at singular points.
\end{example}


\begin{example}
If $\lambda_1\lambda_2 \neq 0$ and $\lambda_1=n\lambda_2, n \in \mathbb Z^+$, i.e. the resonant case, the foliation
is conjugate to the one given by the $1-$ form
$$
nzdw-wdz\; {\mbox{or}}\; (nz+w^n)dw-wdz.
$$
When $\lambda =0 $ (saddle node), according to Dulac, see[1], the vector field can be expressed as
$$
[z(1+\alpha w^k)+wF(z,w)]\frac{\partial}{\partial z}+w^{k+1}
\frac{\partial}{\partial w},
k \in \mathbb N^+, \alpha \in \mathbb C$$
\noindent and $F$ is a holomorphic function vanishing to order $k$.
There is only one separatrix, $(w=0).$
\end{example}
In the following example we describe the case where the foliation admits a meromorphic first integral.

\begin{example}
Let $f$ be a meromorphic function on an algebraic manifold $M$. Then
$\{f=c\}$ are the leaves of a holomorphic foliation. For example, if $M=\mathbb P^2$, then $f=P/Q$ with $P,Q$ homogeneous polynomials without common factor. The leaves $L_c=\{P=cQ\}$ are closed. The singularities contain the discrete set defined by $\{P=Q=0\}$. The form defining the foliation is $QdP-PdQ$ so there might be other singular points. The closure of each leaf intersects the singular set of $\mathcal F$ and  every leaf is a separatrix. Take any leaf $L$, i.e. an irreducible component of some $\{P=cQ\}$. Then $Q=0$  must intersect the algebraic leaf $L$. The  points of intersection must also be a zeroes for $P$ since $P=cQ.$ This gives an example of a dicritical singularity, i.e a singularity with infinitely many separatrices.
In $\mathbb C^2$, it can be shown \cite{CM1982}  that if $\alpha$  is a germ of holomorphic 1-form and f is a germ of holomorhic function near 0, such that $d(\alpha/f)=0$ then $\alpha$ has a multivalued first integral. More precisely, if $f=f_1^{n_1} \cdots f_p^{n_p}$ is the decomposition of $f$ into irreducible factors then $F=\exp \left( \frac{\beta}{f_1^{n_1-1} \cdots f_p^{n_p-1}}\right) f_1^{\lambda_1}\cdots f_p^{\lambda_p}$ is a first integral.
Here the $\lambda_i$ are complex numbers and $\beta$ is a germ of a holomorphic function near 0,

\end{example}

The notion of dicritical singularity is quite central. Given a germ of a holomorphic vector field $Z,$ with isolated singularity at $(\mathbb C^k,0)$, we say that $0$ is $\underline{\mbox{dicritical}}$ if there are infinitely many separatrices through it.
The fact that all leaves are closed in a deleted neighborhood of the origin $(\mathbb C^2,0)$ and that the origin is dicritical
does not imply the existence of a meromorphic first integral. An example of that situation was given by M.
Suzuki. Consider the level sets of the function f(z,w)=xexp(y/x). This function is a first integral, in the complement of the y-axis for the system xdy+(x-y)dx=0. It is easy to check that leaves are closed and that 
the origin is dicritical. But as shown in \cite{S1978} there is no meromorphic first integral in any neighborhood of the origin.

Another interesting observation is given by the following example.The point $0$ is dicritical for $Z=z\frac{\partial}{\partial z}+pw\frac{\partial}{\partial w}, p \in \mathbb N,$ in $\mathbb C^2.$ Indeed the separatrices are $(e^tz_0,e^{pt}w_0), t \in \mathbb C^*.$ Observe that for every $p\in \mathbb N$, $Z$ defines a $\underline{\mbox{linear}}$ foliation and the separatrices are of algebraic degree $p,$ which is arbitrarily large. When there is a meromorphic first integral as in the Example 4 all leaves are separatrices. The converse is not true \cite{CM1982}.  When the first integral is holomorphic, there are only finitely many separatrices and the other leaves are closed.

It can be shown \cite{CM1982} p.29, that if $\alpha$ is a germ of holomorphic 1-form in $(\mathbb C^2,0)$ vanishing exactly to order m at 0, then if $\alpha$ admits m+2 formal separatrices, then $\alpha$ admits uncountably many separatrices. So, in the nondicritical case the number of separatrices is bounded by the order of vanishing at the singularity.

\begin{example}
Ricatti's equation \cite{H1976}, \cite{CL-N1985}. Consider the equation
$$
w'=a(z)w^2+b(z)w+c(z)
$$
\noindent where $a,b,c$ are rational functions on $\mathbb C.$
The form
$$
\omega_0=dw-(a(z)w^2+b(z)w+c(z))dz
$$
\noindent can be extended as a meromorphic form on $\mathbb C \times \mathbb P^1.$
Let $\mathcal P$ be the finite set of poles in $\mathbb C$. It is easy to check that out of ${\mathcal P}$
the form $\omega_0$ is transverse to the fibers $\{z_0\}\times \mathbb P^1$. It follows from a result of
Ereshman that the solutions are coverings of $\mathbb C \setminus \mathcal P.$ In particular
if $\mathcal P = \emptyset$, then the leaves are graphs of meromorphic functions on $\mathbb C.$
When $\# \mathcal P\geq 2$, the leaves are uniformized by the unit disc. Following solutions around a pole of the coefficients induces a map on the space of solutions and hence on $\mathbb P^1,$ it gives clearly a fractional linear transformation:

$$
(z_0,w_0) \rightarrow \frac{w_0A(z_0)+B(z_0)}{w_0 C(z_0)+D(z_0)}.
$$

We get in this way a group of fractional linear transformations on $\mathbb P^1.$
The orbit of a point is the intersection of the leaf through that point with the vertical $\mathbb P^1,$ it is also the orbit of the point under the group of linear transformations.
Hence the dynamics of the foliation can be deduced from the dynamics of the group. 
The Ricatti equation has remarkable extremal properties, see\cite{H1976}  p136. One such propery is: if the differential equation in $\mathbb C^2$, $\frac{dw}{dz} = \frac{Q(z,w)}{P(z,w)},$  is not a Ricatti equation and has a meromorphic solution, then the solution is a rational function.
\end{example}

\begin{example}
Suspension: The simplest examples of laminations are obtained by the process of suspension, that we describe in our context \cite{CL-N1985}.
Let $S$ be a compact Riemann surface of genus $g \geq 2.$ Consider a homomorphism
$h:\Pi_1(S) \rightarrow {\mbox{Aut}}(\mathbb P^1).$ If $\tilde{S}$ denotes the universal covering of $S$ (the unit disc in our case), then $h$ induces an action $\tilde{h}$ on $\tilde{S} \times \mathbb P^1.$
More precisely

\bea
\tilde{h} :\Pi_1(S) &\rightarrow & {\mbox{Aut}}(\tilde{S} \times \mathbb P^1)\\
\tilde{h}[\alpha] (\tilde{b},z) & = & ([\alpha]\cdot \tilde{b}, h[\alpha]\cdot z)\\
\eea
The action of $\tilde{h}$  on $\tilde{S} \times \mathbb P^1$ is free and properly discontinuous. We can then consider the manifold $M=M_h=\tilde{S} \times \mathbb P^1/\tilde{h}$ which is the quotient by the above action.
There is a natural projection $\pi:M \rightarrow S$ and the trivial foliation on $\tilde{S} \times \mathbb P^1$, with leaves $\tilde{S} \times \{z\}$ induces a foliation ${\mathcal F}_h$  
on $M,$ whose leaves are coverings of $S,$ hence hyperbolic Riemann surfaces.
So to any representation $h$ of $\Pi_1(S)$ into PSL$(2,\mathbb C)$ corresponds a foliation. It was observed by E. Ghys \cite{Gh1999} that $M_h$ is a
projective surface and that if $\Lambda$ is the limit set of the group $h(\Pi_1(S))$, the corresponding invariant set in $\mathcal F_h$ can be embedded as a laminated set in $\mathbb P^3.$

This situation has been studied by C. Bonatti and X. Gomez-Mont \cite{BG-M2001}. They have shown that either the group $h(\Pi_1(S))$ has an invariant probability measure or the foliation ${\mathcal F}_h$ 
is uniquely ergodic and this is the generic case. They construct a probability 
measure on $M_h$ by an appropriate averaging process on the leaves and unique ergodicity means that the averaging process applied to an arbitrary leaf gives always the same limit. In these examples however the laminations considered are without singularities. The foliations are transverse to  a fibration with fibre $\mathbb P^1$ or $\mathbb P^2$. See also the recent work by M.Martinez  \cite{M2006}.

\end{example}

\begin{example}
We give a classical example of a Riemann surface lamination of $\mathbb P^3$ without singularities, (Atiyah \cite{A1979}). Let $H$ denote the field of quaternions. Recall that every quaternion $x$ has a unique expression $x=z_1+jz_2, z_1=x_1+ix_2, z_2= x_3+ix_4$ where the $z_j$ are complex numbers. 
Consider the map $\pi: \mathbb P^3(\mathbb C) \rightarrow \mathbb P^1(\mathbb H)$ given by
$[z_1:z_2:z_3:z_4] \rightarrow [z_1+z_2j:z_3+z_4j].$ It is a fibration and each fibre is a complex line in $\mathbb P^3(\mathbb C)$. Indeed  $\mathbb P^1(\mathbb H)$ is not a complex manifold, otherwise it would be a one point compactification of $\mathbb C^2$. This gives an example of a lamination without singularities in $\mathbb P^3$. This lamination is however not a holomorphic foliation.

Clearly one cannot have such an example in $\mathbb P^2$ because of the B\'ezout Theorem. 
It is not possible to have a lamination $(\mathbb P^2,\mathcal L)$ without singularities. Suppose $\mathcal L$ is such a laminaton, then we can write the tangent bundle $\mathbb T \mathbb P^2=T_1\oplus T_2$ where $T_1$ and $T_2$ are two line bundles, $T_1$ is  tangent along the leaves and $T_2$ is normal. On the other hand the topological classification of line bundles is the same as the classification for holomorphic ones on $\mathbb P^2.$ Hence $\mathbb T \mathbb P^2=\mathcal O(p) \oplus \mathcal O(q), p,q \in \mathbb Z.$ 
Since the quotient of a bundle is more positive than the original one we get that
$p\geq 1, q \geq 1.$ The space of sections of $T\mathbb P^2$ corresponds to holomorphic vector fields, it is of dimension $8$. The dimension of the space  of sections of ${\mathcal O}(p) \oplus \mathcal O(q)$ is  $6$, when $p=q=1,$ and larger than $8$ otherwise.

Let $X$ be a compact laminated set in $\mathbb P^k$ laminated by Riemann surfaces. We assume that $X$ contains no germs of higher dimensional complex manifolds. We say that $X$ is real analytic
if $X$ locally is the zero set of a real analytic function $P(z_1,\dots,z_k)$ Setting $w_j=\overline{z}_j$
we can complexify $P$ as a function $R(z_1,\dots, z_k, w_1,\cdots, w_k)$ where $R$ is a holomorphic function.
Assume that this can be done globally. More precisely, we assume that for each $\mathbb C^k$
coordinate chart, we can choose $P$ and $R$ as entire functions. If we fix a germ $L$ of a leaf of the lamination and choose $w^0$ in $\overline{L}$, then by a Segre argument $R(z_1,w_1^0,...,w_k^0)$ vanishes identically
on $L$. In fact in local coordinates the leaf is given by $z_2=z_3=\cdots=z_k=0$ so
$P=z_2r_2+\overline{z}_2s_2+\cdots z_k r_k+\overline{z}_k s_k$ and hence
we see that $R(z_1,\cdots 0)$ vanishes identically. Also  the zero set is a global analytic set in $\mathbb C^k.$ In fact it can be shown \cite{DF1978}
that the leaf is locally the finite intersection of such analytic sets. In particular it follows that each leaf is closed when the lamination is restricted to $\mathbb C^k$.
\end{example}

\begin{example} (Ohsawa \cite{Ohsawa}, Nemirovski \cite{N1999})
Let $s:T \rightarrow \mathbb P^1$ be a meromorphic function on a torus $T$ with simple zeros and poles,
$\{a_j,b_j\}.$ We set $\Sigma=\{(z,rs(z)),z\in T\setminus\{a_j,b_j\}, r \in \mathbb R^*\}\subset T \times \mathbb C^*.$
Then the closure ${\overline{\Sigma}}$ contains the complex curves $\{a_j\}\times \mathbb C^*, \{b_j\}\times \mathbb C^*.$ Let $\rho>1$ and let $T_2$ be the quotient obtained from $\mathbb C^*$ by identifying
$w, \rho^n w.$ The quotient $S$ of ${\overline{\Sigma}}$ in $T \times T_2$ under this identification is a Levi flat
hypersurface and the fibers over $\{a_j,b_j\}$ are copies of $T_2.$ 
It seems difficult to find similar examples in higher dimension due to the fact that meromorphic functions have indeterminacy points.

Another example of a real analytic Levi flat hypersurface in a product of two tori was given
by Grauert. Let $T_1$ be the quotient of $\mathbb C(z)$ by the Lattice $L_1$ generated by $\{1,i\}$ and let $T_2$ be the quotient of $\mathbb C(w)$ by the lattice $L_2$ generated by $\{1,\sqrt{2}i\}.$
Then the set $S=\{ Re(z-w)=0\}$ is a Levi flat hypersurface with dense leaves all biholomorphic to $\mathbb C.$
\end{example}

\begin{example}
Let  $N$ be a compact homogeneous surface. Suppose $K \subset N$ is a nowhere dense compact set, laminated by Riemann surfaces.  Let $\Phi:M \rightarrow N$ be a holomorphic map of maximal rank $2$
from a compact surface M into N. Then after a small sliding, $\Phi^{-1}(K)$ is a laminated compact set in $M.$ To see this, let $X$ be the image of the critical set of $\Phi.$ We can assume that none of the irreducible branches of $X$ is a leaf of $K.$ We can also assume that $K$ does not intersect the singular set of $X.$ Suppose that $X$ is tangent to some leaves. Let $k \geq 2$ be the maximal order of tangency. There can be at most finitely many tangencies of order $k$ of $X$ with leaves.
Next we locally extend the lamination to a lamination of an open set. Notice that small slidings of $X$ can only have one tangency of order $k$ with a leaf in these local flow boxes. Since $K$ is nowhere dense we can arrange that all tangencies are with added leaves. Then we have no tangencies between leaves of $K$ and $X$ of order $k$. Repeating the process we can finally assume that there are no tangencies between leaves of $K$ and $X.$  Then $\Phi^{-1}(K)$ is a nonsingular lamination. If $K$ is a smooth Levi flat hypersurface, then $\Phi^{-1}(K)$ is a also a smooth Leviflat hypersurface. 
\end{example}

\begin{example}
A holomorphic foliation on a complex surface $S^2$ is given in local coordinates by a form
$\alpha=Pdz+Qdw$ where $P,Q$ are holomorphic functions with isolated common zeroes.
If $\Phi:M^2 \rightarrow S^2$ is a holomorphic map of maximal rank $2$, then $\Phi^*(\alpha)=Fdu+Gdv$ defines a holomorphic foliation with isolated singularities which becomes evident after dividing
out common factors of $F,G $. This permits to construct singular foliations on some tori.
\end{example}

\subsection{Singularities for foliations of dim 1 in $\mathbb P^k$}

\begin{example}
The Lorenz vector field in $\mathbb R^3$ is an example of a quadratic polynomial vector field with a complicated
orbit structure.
$$
Z=\sigma(y-x)\frac{\partial}{\partial x}+(rx-y-xz)\frac{\partial}{\partial  y}+(xy-bz)\frac{\partial}{\partial z}.
$$
The Lorenz attractor is given by the parameter choices $\sigma=10, b=8/3, r=28.$ For a proof that this is a strange attractor as indicated by computer pictures, see Warwick \cite{W2002}.
The complexification is given on $\mathbb P^3([t:x:y:z])$ as
$$
X=\sigma t(y-x)\frac{\partial}{\partial x}+
(rxt-yt-xz)\frac{\partial}{\partial y}+(xy-bzt)\frac{\partial}{\partial z}.
$$
This has three singular points in the chart $t=1$ which are all real and equal
to $P=[1:0:0:0]$ and $C^\pm = [1:\pm \sqrt{b(r-1)}:\pm \sqrt{b(r-1)}:r-1]$. The singularity at $P$ is a hyperbolic saddle with a two dimensional stable manifold containing the $z$ axis and an unstable curve. The singularity at $C^\pm$ have an expansive pair of complex eigenvalues and a stable eigenvector. At infinity, $t=0,$ the singularity set contains two pieces, the curve
$t=x=0$ and the point $[0:1:0:0].$ The eigenvalues at the latter are $0, \pm i.$
\end{example}

\begin{example}
Let $H(z_1, \dots,z_k,w_1,\dots,w_k)$ be a holomorphic polynomial in $\mathbb C^{2k}.$
Then there is associated to $H$ a holomorphic Hamiltonian vector field:
$$
Z= -\sum_j \frac{\partial H}{\partial w_j}\frac{\partial}{\partial z_j}+\sum_j \frac{\partial H}{\partial z_j}\frac{\partial}{\partial w_j}.
$$
This gives a Riemann surface lamination of $\mathbb P^{2k}$. Each leaf is contained in a level set of $H$. More generally one can define a multivalued Hamiltonian of the form
$H=\Pi_j H_j^{\alpha_j}$ for polynomials $H_j$ and complex numbers $\alpha_j.$ In $\mathbb C^2$ this has the zero sets of the $H_j$ as algebraic integral curves. 
\end{example}

We discuss some basic properties of holomorphic foliations by Riemann surfaces in $\mathbb P^k.$ Suppose $\mathcal F$ is a holomorphic foliation of dimension one in $\mathbb P^k,$ i.e. locally the leaves of $\mathcal F$ are tangent to a holomorphic vector field. More precisely there is a covering $(U_i)$ of $\mathbb P^k$ with holomorphic vector fields $X_i$ on $U_i$ satisfying the compatibility condition $X_i=f_{ij}X_j$ on $U_i\cap U_j$ with $f_{ij}$ nonvanishing holomorphic functions. Let $\pi:\mathbb C^{k+1}\setminus \{0\} \rightarrow \mathbb P^k$ denote the canonical projection. It can be shown that $\pi^* \mathcal F$ can be defined by a holomorphic vector field on $\mathbb C^{k+1}$

$$
Z=\sum_{i=0}^k F_i(z) \frac{\partial}{\partial z_i}
$$

\noindent where the $F_i$ are homogeneous polynomials of degree $d \geq 1.$
We call $d$ the degree of the foliation. This is the analogue of Chow's theorem that any analytic set in $\mathbb P^k$ is algebraic.
In the chart $(z_0=1)$ the vector field is given by

$$
Z_0=\sum_{i=1}^k (F_i(1,z)-z_iF_0(1,z))\frac{\partial}{\partial z_i}.
$$
Conversely the projections by $\pi$ of the trajectories of such a vector field gives a holomorphic foliation of $\mathbb P^k.$  Indeed if $z(\zeta)=(z_0(\zeta), \dots , z_k(\zeta))$ satisfies $z'_i(\zeta)= F_i(z(\zeta)), 0 \leq i \leq k$, then $Z(\eta):=cz(\eta c^{d-1}), c\in \mathbb C,$ is also an integral curve which induces the same trajectory in $\mathbb P^k.$ We can always assume that the $F_j's$ have no common factor. The space of such foliations of degree $d$ is denoted by $\mathcal F_d(\mathbb P^k),$ it is clearly a Zariski open set in some projective space.
 The space of polynomials of degree d in k variables is of dimension $ \frac{(d+k)!}{k!d!}$. Using the representation of $$
Z_0=\sum_{i=1}^k (F_i(1,z)-z_iF_0(1,z))\frac{\partial}{\partial z_i},
$$
we can compute the dimension N of the space of holomorphic foliations of degree d. We have $N=(d+k+1) \frac{(d+k-1)!}{(k-1)!d!}-1$ \cite {Br2000}. It contains the subspace of foliations obtained by pull-back of a foliation of degree 1 by a holomorphic endomorphism of algebraic degree $\ell $. The pullback of a form of degree 1, by an algebraic endomorphism of degree $\ell $, is generically of degree $2\ell-1.$ The dynamics of such foliations is easy to describe..

A point $p\in \mathbb P^k$ is a singularity for $\mathcal F$ if $F(p)$ is colinear with $p$, i.e. if $p$ is a pole or a fixed point of $f=[F_0:\cdots:F_k]$ as a meromorphic map in $\mathbb P^k.$ If $f$ is non holomorphic, then the indeterminacy set is analytic of codimension $\geq 2$, it can be of positive dimension when $k \geq 3.$ Assume that $f$ is holomorphic. To count the fixed points it is enough to apply the B\'ezout theorem to the equations $F_j(z)-t^{d-1}z_j=0$ in $\mathbb P^{k+1}$ with homogeneous coordinates $[z:t]$ and observe that $[0:\cdots:0:1]$ is a solution. The number of fixed points counted with multiplicity is $\frac{d^{k+1}-1}{d-1}$ so the singularity set, Sing$(\mathcal F)$, of any holomorphic foliation is non empty. It is also clear that for an open Zariski dense set $\mathcal U$ in $\mathcal F_d(\mathbb P^k),$
Sing$(\mathcal F)$ is discrete. The eigenvalues at  a singularity are given by the eigenvalues of f-Id, at the corresponding fixed point.
We are going to recall some results on the local classification of singular points. Let $X$ be a germ of a holomorphic vector field near $0\in \mathbb C^k$, with an isolated singularity at $0.$ The matrix $DX(0)$, i.e. the linear part at $0$, has eigenvalues $\lambda_1,\dots,\lambda_k.$

\begin{definition}
The singularity is in the Poincar\'e domain if the convex hull
in $\mathbb C$ of $\{\lambda_1,\dots,\lambda_k\}$ does not contain the origin, it is in the Siegel domain otherwise.
\end{definition}

\begin{definition}
The singularity of $X$ at $0$ is hyperbolic if for all $i,j$ the $\lambda_i \neq 0$ and the
$\frac{\lambda_i}{\lambda_j}$ are non real.
\end{definition}

Clearly in dimension $2$ a hyperbolic singularity is always in the Poincar\'e domain.

The following result is due to Chaperon \cite{Ch1986}.

\begin{theorem}
(Chaperon)Let $X$ be a germ of a holomorphic vector field in $(\mathbb C^k,0)$. If $0$ is a hyperbolic singularity, then $X$ is topologically linearizable.
\end{theorem}

This means that there is a homeomorphism $\Phi:(\mathbb C^k,0)\rightarrow (\mathbb C^k,0)$ sending the foliation defined by $X$ to the foliation defined by the vector field $X_0=\sum_i \lambda_i z_i \frac{\partial}{\partial z_i}$ where the $\lambda_i$ are the eigenvalues of $DX(0)$. 

If we write the formal conjugation of a holomorphic vector field near a singularity, to its linear part, one has to divide by the quantities
$(\lambda,m)-\lambda_j,$ $m=(m_1,\dots,m_k)\in (\mathbb Z^+\cup \{0\})^k,
|m|=\sum m_j \geq 2.$ Here $(\lambda,m)$ denotes the inner product $\sum \lambda_j m_j.$ To prove convergence
one assumes that these quantities are non zero and not too close to zero.

We define the resonances of $\lambda\in \mathbb C^k$ as:
$$
\mathcal R=\{(m,j); m=(m_1,\dots,m_k), m_j \geq 0, |m| \geq 2,
(\lambda,m)-\lambda_j=0\}.
$$

The set $\{(\lambda,m)-\lambda_j; |m| \geq 2\}$ has zero as a limit point if and only if $\lambda$ belongs to the Siegel domain.

An example of a normal form, when $\mathcal R$ is nonempty is given above in 
Example 3.
\medskip

We recall two results from the local theory near a singular point in any dimension.\\

\begin{theorem}
(Poincar\'e) \cite {AI1988} A germ of a singular holomorphic vector field in $(\mathbb C^k,0)$ with a non resonant linear part (i.e. $\mathcal R$ is empty)
such that $\lambda$ is in the Poincar\'e domain is holomorphically equivalent to it's linear part.
\end{theorem}

To get linearization for $\lambda$ in the Siegel domain it suffices to assume the so called Brjuno condition: condition (B) \cite{AI1988},\cite{B1979}. Define for $n\in{\mathbb N }$
$$
\Omega(n)=\inf_{1\leq j \leq k} \{|(\lambda,m)-\lambda_j|,
m=(m_1,\dots, m_k), |m| \leq 2^{n+1}\}.
$$

Condition $(B)$ is satisfied when

$$
\sum_{n \geq 1} \frac{\log (1/\Omega(n))}{2^n} <\infty.
$$

Implicitly we are assuming that there are no resonances. In dimension $2$, if $\lambda\in \mathbb R^-$, the Brjuno condition is $\sum_{n \geq 1} \frac{\log q _{n+1}}{q_n}<\infty$ where $\{p_n/q_n\}$ is the nth
approximant of $-\lambda.$ 

\medskip

\begin{theorem} (Brjuno) \cite{AI1988}\cite{B1979}A germ of a singular holomorphic vector field in $(\mathbb C^k,0)$ with non resonant linear part and which satisfies
the Brjuno condition is holomorphically linearizable.
\end{theorem}

\subsection{The space of holomorphic foliations on $\mathbb P^k$}

Let $\mathcal F_d(\mathbb P^k)$ denote the space of holomorphic foliations of degree $d$ by Riemann surfaces in $\mathbb P^k.$ We parametrize this space by $\mathcal M(\mathbb P^k)=[F_0:\cdots :F_k]$ where the $F_j$ are homogeneous holomorphic polynomials of degree $d$, with no common factor and hence with indeterminacy set of codimension $\geq 2.$ 
The space $\mathcal F_d^0$ of foliations with discrete singularity set is an open Zariski dense set and is parametrized by meromorphic self maps with discrete indeterminacy set and with a discrete set of fixed points.

In dimension $2$, the number of tangencies of a generic line with the foliation of degree $d$  is $d.$ In higher dimension the tangencies of a foliation of degree d with a generic hyperplane is  an algebraic hypersurface  of degree d in the hyperplane.

Let $\mathcal H^y_d(\mathbb P^k)$ denote the space of foliations of degree $d$ such that all singular points are hyperbolic. 

\begin{proposition}
The set $\mathcal H^y_d(\mathbb P^k)$ contains an open real Zariski dense set in $\mathcal F_d(\mathbb P^k).$
\end{proposition}

We leave the details to the reader.

The Poincar\'e problem is to bound the degree of invariant algebraic curves of a foliation of degree $d.$

 Carnicer \cite{C1994} has proved the following:
 
 \begin{theorem}
 Let ${\mathcal F}\in \mathcal F_d(\mathbb P^2)$ and let $A$ be an algebraic curve invariant by $\mathcal F.$ If there are no dicritical singularities on $A,$ then ${\mbox{deg}}(A) \leq d+2.$ The estimate is sharp.
 \end{theorem}

The case of invariant curves in $\mathcal F_d(\mathbb P^k)$ is studied by Lins-Neto, \cite{LN2002}.
 He has shown that even under a natural restriction, "fixed local analytic type",  it is not possible to bound the degree of algebraic solutions, see \cite{LN2002} for a precise statement.

The following theorem of E. Ghys \cite{Gh2000}, which improves a theorem by J. P. Jouanolou \cite{J1978}, gives a general result on closed leaves.

\begin{theorem}
(Ghys) Let $\mathcal F$ be a codimension $1$ holomorphic foliation (possibly singular) on a compact complex manifold. Then $\mathcal F$ has only a finite number of closed leaves except when $\mathcal F$ admits a meromorphic first integral, in which case all leaves are closed.
\end{theorem}

Here we just prove a weaker elementary estimate on the total degree of invariant curves,  in a generic case. 

\begin{proposition}
Let $(A_i)$ denote the algebraic curves invariant under $\mathcal F \in \mathcal H^y_d(\mathbb P^2).$
Then $\sum_i {\mbox{deg}}(A_i) \leq 2(d+2).$ 
\end{proposition}

\begin{proof}
Let $A=(h=0)=\cup_{i \in I} A_i$ (or any finite union if there are infinitely many). Since the singularities are hyperbolic, if $p\in {\mbox{Sing}}(\mathcal F)
\cap A,$ then $A$ is a separatrix and if we consider $A$ as a branched covering in generic coordinates, the local degree of the covering near a singularity is at most $2.$ Moreover at all non singular vertical points, $A$ is tangent to second order to its tangent line. So we can choose a $\mathbb C^2$ chart
containing all the singular points and consider that $A$ is a branched covering over the $z-$ axis. We can assume that $A$ is transverse to the line at infinity $L_\infty$ and that it omits the points $[0:1:0], [1:0:0].$ So near $L_\infty,$ $A=\cup_j\{w=\alpha_j z+ o(z)\}$ for distinct $\alpha_j$. The foliation is associated in that chart to a polynomial vector field $P\frac{\partial}{\partial z}+Q\frac{\partial}{\partial w}.$ 
with ${\mbox{deg}}(P) \leq d+1.$  and ${\mbox{deg}}(Q) \leq d+1.$
We count the number of vertical points on $A \setminus {\mbox{Sing}}(\mathcal F)$ in two ways. Let $h=0$ be an equation of $A.$ The vertical points are given by
the equations  $P=0,h=0.$ If $h$ is  of degree $N=\sum_{i\in I} {\mbox{deg}}(A_i)$, we get at most $N(d+1)$ nonsingular vertical points, besides the singular points. So finally the number of such points is bounded by $N(d+1)+2(d^2+d+1).$ The other way is to count the number of common zeroes of $h$ and $\frac{\partial h}{\partial w}$ which gives $N(N-1).$ 
By the B\'ezout Theorem there are at most $(d+1)^2$ common zeroes of $P,Q$, i.e. singular points. Each singular point can contribute at most $2$ common zeroes of $h$ and $\frac{\partial h}{\partial w}$ counted with multiplicity.Because of hyperbolicity, the multiplicity  of the zeroes  
at the singular points is bounded by $2.$ Hence the number of non singular vertical points is at least $N(N-1)-2(d+1)^2.$ Hence $N(N-1)-2(d+1)^2\leq N(d+1).$ 
 So $N\leq 2(d+2).$

\end{proof}

\begin{remark}
The argument works as long as we assume that there are only less than $n(d)$ separatrices at every singular point and that their local multiplicity as branched covers are bounded by some $m(d)$ in which case we find
$\sum {\mbox{deg}}(A_i) \leq c(d).$
\end{remark}

 Jouanolou proved the important fact that not all foliations of $\mathbb P^2$ have algebraic leaves.
\begin{theorem}(Jouanolou \cite{J1979}, Lins-Neto, Soares, \cite{LNSo1996}).
There is a real Zariski dense open set $\mathcal H(d)\subset \mathcal F_d(\mathbb P^k)$ such that for every $\mathcal F \in \mathcal H(d),$\\
i) all the singularities of $\mathcal F$ are hyperbolic,\\
ii) $\mathcal F$ has no invariant algebraic curve.
\end{theorem}

\begin{proof}
(Sketch) We consider $\mathcal H^y_d(\mathbb P^2)$, the space of foliations  
of degree $d$ with only hyperbolic singularities. According to Proposition 2 an algebraic leaf is necessarily of degree $\leq c(d).$ Let $\mathcal A_\ell$ denote the compact space of (possibly reducible) curves of degree $\ell \leq c(d).$
 We consider
 
 $\Gamma_\ell=\{(\mathcal F,L)\in \mathcal H^y_d(\mathbb P^2) \times \mathcal A_\ell,\;{\mbox{each component of}}\;
L \;{\mbox{is a leaf of }}\; \mathcal F\}.$

The set $\Gamma_\ell$ is analytic, hence the projection is analytic in $\mathcal H^y_d(\mathbb P^2).$ So it is enough to exhibit an example of a foliation in $\mathcal H^y_d(\mathbb P^2)$, without an algebraic solution. The following basic example is due to Jouanolou \cite{J1979}, it was extended to $\mathbb P^k$ by
Lins-Neto Soares \cite{LNSo1996}. It is defined as the foliation associated to the form:
$$
\omega_0=(z^{d-1}t-w^d)dz+(w^{d-1}z-t^d)dw+(t^{d-1}w-z^d)dt.
$$

We don't have to check that the above example has hyperbolic singularities since
if the projection is not surjective, the complement of the image intersects $\mathcal H^y_d(\mathbb P^2)$.
\end{proof}

\subsection{Minimal sets}

\begin{definition}
A minimal set for $(X,\mathcal L,E)$ in $M$ is a compact subset $Y \subset X$ such that $Y$ is not contained in $E,$ $Y\setminus E$ is a union of leaves and for every leaf $L \subset Y,  \overline{L}=Y.$
\end{definition}

It is easy to show that if there is a neighborhood $V$ of $E$ so that no leaf is contained in $V,$ then minimal sets exist.  Let $V$ be an open neighborhood of $E,$ the singular set, such that $V$ contains no leaf. Let $X_\alpha$
be a decreasing chain of closed invariant sets for the foliation. Then $\cap (X_\alpha \setminus V)$ is nonempty. Hence $\cap X_\alpha$ is not reduced to $E$ and Zorn's Lemma implies the existence of a minimal set. We will sometimes write minimal set with singularities to emphasize that we allow singularities in $Y.$ It is indeed not known if there is a minimal set without singularities in $\mathbb P^2.$
The problem is discussed in \cite{BLM1992}, \cite{CLS1988}, \cite{H1976}.
In particular, minimal sets always exist for holomorphic foliations by Riemann surfaces in $\mathbb P^k$ with isolated singularities.The sets X(c) in example 2 are minimal sets since every leaf  is dense in X(c).

The global properties of holomorphic foliations of degree $d$ by Riemann surfaces in $\mathbb P^2$ are still poorly understood. However the case where the foliation is tangent to the line at infinity $L_\infty$
has been analyzed by Il'yashenko \cite{Il1987}: 

\begin{theorem} Let $\mathcal F_d(\mathbb C^2)$ denote the foliations of $\mathbb P^2$
of degree $d,$ which are tangent to the line at infinity.
There is a set of full Lebesgue measure $\mathcal A_d \subset \mathcal F_d(\mathbb C^2)$ such that for any $\mathcal F \in \mathcal A_d$ every leaf in $\mathbb C^2$ is dense. Any measurable set of leaves of $\mathcal F\in \mathcal A_d$ has zero or total Lebesgue measure.
\end{theorem}

Mjuller \cite{M1975} has constructed a non empty open set $\mathcal U$ of foliations in $\mathcal F_d(\mathbb P^2),$ 
such that for $\mathcal F \in \mathcal U$ every leaf in $\mathcal F$ is dense.
Loray-Rebelo \cite{LR2003} have constructed open sets $\mathcal U_k$ in $\mathcal F_d(\mathbb P^k)$ with the same property. They also show that for $\mathcal F \in \mathcal U_k$ any measurable set of leaves is of zero or full Lebesgue measure.

Example 2 describes a class of minimal sets with singularities. For the open set of foliations constructed by Mjuller and Loray-Rebelo, the minimal set is $\mathbb P^k.$ We will see that for a dense real Zariski open set of foliations in $\mathcal F_d(\mathbb P^2)$, then the minimal set is unique and that it enjoys properties of unique ergodicity.

\section{Directed positive closed currents}

Let $(X,{\mathcal L}, E)$ be a lamination with singularities by Riemann surfaces. We assume $X$ is contained in a compact K\"{a}hler manifold $ (M,\omega)$ of dimension $k.$ The K\"{a}hler condition is most of the time unnecessary. A Hermitian metric is enough. We will assume that $\Lambda^2(E)=0$
where $\Lambda^2$ denotes the two dimensional Hausdorff measure. The lamination is defined by continuous $(1,0)$ forms $\gamma_j,j=1,\dots,\ell$ on $M,$ such that $\gamma_j \wedge [V_\alpha]=0$ for every plaque $V_\alpha;$ $[V_\alpha]$ denotes the current of integration on the plaque $V_\alpha.$ The $\gamma_j$ vanish on $E$ and  rank $(\gamma_j)
=k-1$ on $X \setminus E.$

\begin{definition}
A positive current $T$ of bidimension $(1,1)$ on $M$ is weakly directed by $(X,\mathcal L,E)$ if
$T \wedge \gamma_j=0, 1 \leq j \leq \ell$ and $T$ is supported on $X.$ 
\end{definition}

The following is a stronger condition.

\begin{definition}
A positive harmonic current $T$ of bidimension $(1,1)$ on $M$ is directed by
$(X,\mathcal L,E)$ if $T$ is supported on $X$ and if
in local flow boxes away from the singularities $T$ has the form
$$
T=\int_\Delta h_\alpha [V_\alpha]d\mu(\alpha)
$$

\end{definition}

Here $\mu$ is a positive Borel measure on a transversal $\Delta$  and $[V_\alpha]$ denotes the current of integration on the plaque $V_\alpha.$
Clearly, if $T$ is directed, then $T$ is weakly directed. We will consider the  question: under which conditions is
a weakly directed $T$ directed?

The following result is proved in \cite{S1976}  when $E$ is empty.

\begin{theorem} (Sullivan) Let $(X,\mathcal L,E)$ be as above. Either there is a non zero positive closed $(k-1,k-1)$
current weakly directed by ${\mathcal L}$ or there is a smooth exact form of degree $2$ which is strictly positive on positive weakly directed $(k-1,k-1)$ currents.
\end{theorem}


\begin{proof}
Let $\mathcal C$ denote the convex compact set defined by
$$
\mathcal C:=\{T \; {\mbox{on}} \; X;T\geq 0,\; T\;{\mbox{has bidimension (1,1)}}, T \wedge \gamma_j=0, \int T \wedge \omega=1\}.
$$
Let $F:=\{d\phi\}^\perp $ denote the orthogonal in the space of currents of bidimension
 $(1,1)$ of the exact smooth forms of degree $2$.Indeed F is the space of closed currents. If $\mathcal C \cap F \neq \emptyset$ then there exists
 a non trivial positive closed current $T$ weakly directed by $\mathcal L.$ Necessarily $T$ has no mass on $E$ since $\Lambda^2(E)=0,$ and positive closed currents give no mass to such sets. This is a consequence of Federer's support theorem \cite{F1969}.

 \medskip
 
 If $\mathcal C \cap F=\emptyset,$ we apply the Hahn-Banach Theorem. The space of currents is reflexive, hence by compactness of $\mathcal C$, 
 there exists an element $\ell \in \overline{\{d\phi\}},$ 
 such that $\ell(T)>\delta>0, T \in \mathcal C$ and $\ell$ vanishing on $F^\perp.$ We can assume 
 $\ell=d\phi.$ 
 
 \end{proof}
 
 \begin{remark}
 In order to prove the existence of a weakly directed  closed current, the hypothesis on Hausdorff dimension can be replaced by the weaker assumption that there is no nonzero closed current supported by $E$. For example we can assume $M \setminus E$ is not $k-1$ pseudoconvex in which case $E$ cannot support a current of bidimension (1,1), see \cite{FS1995}. For $k=2$, this means that $M \setminus E$ is not pseudoconvex, and if $M=\mathbb P^2$ this is equivalent to $M \setminus E$ is not Stein.
 \end{remark}

 The following theorem,  due to Sullivan \cite {S1976},  explains why  weakly directed closed currents are useful.
 The theorem applies when there are no singularities and the lamination is smooth enough.
Also the result works in any dimension of the ambient manifold, but the leaves are Riemann surfaces.

 \begin{theorem} (Sullivan)
 Every positive closed weakly directed current $T$ on a lamination  by Riemann surfaces which is transversally ${\mathcal C}^1$
 has the following form in a flow box,
 $$
 T=\int [V_\alpha]d\mu(\alpha)
 $$
 \noindent where $\mu$ is a holonomy invariant positive measure. Any holonomy invariant 
 measure corresponds to a positive closed directed current.
 \end{theorem}
 
 \begin{proof} We show the last part of the Theorem first.
 Recall that a measure $\mu$ is holonomy invariant if $\mu(\gamma(A))=\mu(A)$ for every Borel set $A$ on a transversal and $\gamma$ is the sliding along leaves from a transversal to another one in a flow box. More formally let ${\mathcal U}=(U_i)_{i\in I}$ be an atlas. Let $\Gamma$ denote the pseudogroup of holonomy transformations acting on transversals of ${\mathcal L},$ then $\mu(\gamma(A))=\mu(A)$ when defined. Let $\mu$ be a transverse measure. For a smooth $(1,1)$ form $\phi$ we define
 $$
 <T_\mu,\phi>=\sum_{i\in I}\int_{S_i}<[V_\alpha],\lambda_i \phi>d\mu_i(\alpha).
 $$
 Here $(\lambda_i)$ is a partition of unity associated to the covering $(U_i), S_i$ is a transversal in $U_i$, and $<[V_\alpha],\lambda_i\phi>$ denotes the integration on the plaque $V_\alpha.$ It is easy to check that the definition is independent of the partition of unity. Hence $T_\mu$ is closed.
 
 \medskip
 
 We next show that a positive closed weakly directed current is defined by a holonomy invariant measure.
 It is sufficient to analyze the structure of positive closed weakly directed currents in a flow box $B$ of a ${\mathcal C}^1$ lamination.
 
 \medskip
 
 Let $S$ be a transversal and denote by $\pi$ the projection on the transversal $S$. We assert that the extreme rays of the convex cone of weakly directed positive closed currents are given by integration on plaques: Let $\chi$ be a smooth function with compact support on $S$. Since $\pi$ is ${\mathcal C}^1$ then $d((\chi \circ \pi)T)=\chi'(\pi) d\pi \wedge T.$ But $d \pi$ is pointwise in the span of $(\gamma_j,\overline{\gamma}_j)$ which defines the lamination so $d\pi \wedge T =0.$ Hence $(\chi \circ \pi )T$ is a weakly directed closed current. The assertion follows. The representation of $T$ given in the theorem is then a consequence of Choquet's integral representation theorem.

 \end{proof}

  For laminations by Riemann surfaces in higher dimension, the theorem is false
if we drop the condition of transverse smoothness. In fact the laminated set given by the union of the complex curves $\gamma_t(s)=(s,(s-t)^2,(s-t)^3)$ have the property that the current of integration of the $z$ axis is weakly directed but not directed, see \cite{FWW2007a}. Transversally, the above lamination is only H\"{o}lder continuous. Also this paper gives a proof of the theorem in the case of surfaces. In this case the lamination is given by holomorphic motion and this gives enough transverse regularity to use an interpolation argument.

 \begin{theorem} (Plante) Let $(X,\mathcal L,E)$ be a lamination by Riemann surfaces in a compact complex manifold $M.$ Suppose there is a sequence of relatively open smoothly bounded domains $D_n \subset\subset L_n$, $L_n\subset X$ Riemann surfaces weakly directed by $\mathcal L,$ such that $\frac{{\mbox{vol}}(\partial D_n)}{{\mbox{vol}}( D_n)}\rightarrow 0.$ Then there
 is a weakly directed  closed current $T$ supported on  $\overline{\cup D_n}.$  It the $L_n$ are locally contained in leaves outside of $E,$ then $T$ is directed by the lamination.
 \end{theorem}
 
 \begin{proof}
 Define for a test form $\phi$:
 $$
 T_n(\phi):= \frac{1}{{\mbox{vol}}(D_n)}\int_{D_n}\phi.
 $$
 Then 
 \bea
 |T_n(d\eta)| & = & \left| \frac{1}{{\mbox{vol}}(D_n)}\int_{\partial D_n}\eta\right|\\
 & \leq & \|\eta\|_\infty 
\frac{{\mbox{vol}}(\partial D_n)}{{\mbox{vol}}( D_n)}\rightarrow 0\\
\eea

Clearly $(T_n)$ has bounded mass.
So any cluster point $T$ of the sequence $(T_n)$ is a closed current weakly directed by the foliation. 
Assume next that $L_n$ is locally contained in leaves outside of $E.$ Notice that if a boundary component of $D_n$ belongs to one plaque, then we can increase $D_n$ by filling in the hole. Filling in all such local holes increases $vol(D_n)$ and decreases $vol(\partial D_n)$ . We assume below, that the local holes have been  filled. We also shrink the plaques.

To show that $T$ is directed, note that in any flow box the current $T_n$ is given by a finite number of plaques. It the total mass of the plaques in the flow box goes to zero, the limit $T$ is identically zero there.
If the mass does not go to zero, the fraction of plaques that is only partially contained in $D_n$ goes to zero. This is because if a plaque is only partially contained in the shrunk $D_n$, there is a boundary piece with length at least comparable to the area of the plaque. Hence when we pass to limits we can suppress the plaques which are not
completely contained in $D_n$. Hence the limit is directed.

\end{proof}

 \begin{remark}
 The previous result implies that when a lamination has no directed positive closed current,
 then leaves satisfy the following growth condition. Given a leaf $L,$ fix $p\in L$ and consider
 $D(r)$, the disc of center $p$ and radius $r,$ on the leaf, with respect to a fixed Hermitian
 metric on the manifold. Then Plante's theorem implies that
 $$
 \underline{\lim}_{r \rightarrow \infty} \frac{{\mbox{vol}}(\partial D(r))}{{\mbox{vol}}( D(r))}
 =\alpha_L>0.
 $$
 In particular there is a constant $C=C_L$ such that
vol$(D(r))\leq C$ vol$(\partial D(r)).$
 It follows that the areas of the leaves grow at least exponentially
\cite{MP1974}.
 \end{remark}

 \begin{corollary}
 Let $(X,\mathcal L,E)$ be a laminated compact in $M.$ Let $\phi:\mathbb C \rightarrow X $ be a non constant holomorphic map. Assume $$<\gamma_i(\phi(\zeta)),\phi'(\zeta)>=0,\; \zeta \in \mathbb C.$$ Then there is a nontrivial weakly directed positive closed current $T,$ for the lamination $(X,{\mathcal L},E)$ defined by the $(\gamma_i).$ If in addition we make the hypothesis that $\phi(\mathbb C)$ is locally contained in a leaf outside of $E$, then there is a nontrivial positive closed directed current.
 \end{corollary}
 
\begin{proof}
We recall that the $\gamma_j$ vanish on $E.$
Let $D_r=\phi(\Delta_r)$ where $\Delta_r$ denotes the disk in $\mathbb C$ centered at zero with radius $r.$ Let $A(r):= {\mbox{Area}}(D_r), L(r):= {\mbox{Length}}(\partial D_r).$ It is enough, by the weakly directed case of Plante's theorem, to show the existence of a sequence $r_i \rightarrow \infty$ such that $\lim_{i} \frac{L(r_i)}{A(r_i)}\rightarrow 0.$ Suppose on the contrary that $\underline{\lim}_{r \rightarrow \infty} \frac{L(r)}{A(r)}=C>0.$ We will get a contradiction using an inequality due to Ahlfors. We have by Schwarz's inequality

$$
L(r)=r\int_0^{2\pi} |\phi'(re^{i\theta})|d\theta \leq r\left(\int_0^{2\pi} |\phi'(re^{i \theta}|^2d\theta\right)^{\frac{1}{2}}\sqrt{2\pi}.
$$

On the other hand, 

$$
\frac{dA(r)}{dr}=r\int_0^{2\pi}|\phi'(re^{i \theta})|^2d\theta.
$$

So $L^2(r) \leq 2\pi r \frac{dA(r)}{dr}.$ For each $c>0,$ define $E_c:=\{r; \frac{L(r)}{A(r)}\geq c\}.$ 
We get on $E_c$

\bea
\frac{dA}{dr} & \geq & \frac{L^2(r)}{2\pi r}\\
& \geq & \frac{c^2}{2\pi r}{A^2(r)}\\ 
& \Rightarrow &\\
\frac{A'}{A^2} & \geq & \frac{c^2}{2\pi r}, r \in E_c\\
& \Rightarrow &\\
\int_{E_c} \frac{dr}{r} & \leq & \frac{2\pi}{c^2}\int_{E_c}\frac{A'}{A^2}\\
&  \leq & \frac{2\pi}{c^2}\int_{0}^\infty \frac{A'}{A^2}\\
& \leq & \frac{2\pi}{c^2 A(0)}.\\
\eea

So $E_c$ cannot contain a neighborhood of infinity. The other assertions follow from the directed case of Plante's theorem.

\end{proof}

\begin{example}
Let $X_c=\overline{\{(z,w); |z|=c|w|^\lambda\}}$ with $\lambda \in \mathbb R,$ where the closure
is in $\mathbb P^2$. Then $X_c$ has many images of $\mathbb C$, $\phi(\zeta)=
(ce^{\lambda \zeta}, e^\zeta e^{i\theta}).$ Let $\pi$ denote the projection from $\mathbb C^3\setminus 0
\rightarrow \mathbb P^2.$ The current $T$ defined by $\pi^*T=i\partial \overline{\partial} u,$
with $u(z,w,t)= \log (\max\{|z||t|^{\lambda-1},c|w|^\lambda\}),$ (assuming $ \lambda>1$) is directed by the lamination.
If $\lambda$ is irrational then it can be shown that it is up to multiplication by a constant the unique positive closed current supported by $X_c.$ The holomorphic foliation of $\mathbb P^2$ associated to $\alpha=\lambda zdw- wdz, \lambda \in \mathbb R$ admits a continuous family of directed positive closed currents. In the presence of a transverse measure, i.e. of a positive closed directed current, the dynamics is quite clear. But most foliations do not admit a directed positive closed current.

We can construct similar examples for foliations of higher degree. Let $F$ be a holomorphic map of algebraic degree $d$. Then for generic F, $F ^*\alpha$ defines a holomorphic foliation of degree $2d-1$, which has also infinitely many positive closed directed currents.
\end{example}

\begin{proposition}
Let $(X,\mathcal L,E)$ be a lamination with singularities in a compact complex manifold $\mathcal M$.
Assume $\Lambda^1(E)=0$ or $\Lambda^2(E)=0$ and $E$ is locally complete pluripolar.
Suppose $L$ is a leaf which is closed in $X \setminus E,$ then $\overline{L}$ is an analytic set.
\end{proposition}

\begin{proof}
Consider $[L]$, the current of integration on $L.$ It is a positive closed current of bidimension $(1,1)$ in $\mathcal M\setminus E.$ Under the hypothesis it follows that the trivial extension $\tilde{[L]}$ of $[L]$ is still positive and closed. If $\Lambda^1(E)=0$ this is a consequence of Harvey's theorem \cite{H1974}, otherwise we use \cite{FS1995} to show that the mass of $L$ is bounded near E, a version of Bishop's Theorem \cite{De 1997}  gives the result.

Let $\nu(\tilde{[L]},a)$ denote the Lelong number at $a.$ According to Siu's theorem \cite{S1974}

$$
L' =\{a; \nu(\tilde{[L]},a) \geq 1\}
$$

\noindent is an analytic set in $\mathcal M$, which contains $L$. Hence
$\overline{L}=L'$ and $\overline{L}$ is algebraic.

\end{proof}
\medskip

\begin{theorem}
Let $(X,\mathcal L,E)$ be a $\mathcal C^2$ lamination with singularities in $\mathbb P^2.$ If a leaf is contained in an algebraic curve $A$ and if  X contains a neighborhood of A, then $A \cap E \neq \emptyset.$
\end{theorem}

\begin{proof}
 The proof is by contradiction, so assume that $A \cap E=\emptyset.$
We introduce first the conormal bundle associated to a lamination $(X,\mathcal L,E),$ transversally $\mathcal C^2$ on a surface $M.$ The construction is classical.

We can cover $X \setminus E$ by a locally finite collection of charts $(U_i)$ with $\phi_i:U_i \rightarrow \Delta \times T_i,
{\mathcal C}^2$ homeomorphisms. The change of coordinates is of the form
$z_j=h(z_i,t_i), t_j=t_{j}(t_i).$ The plaques are given by $t_i= c^{\mbox{te}}.$ In particular $dt_j= t_{ij}dt_i$ with $t_{ij}$ constant on plaques
in $U_i \cap U_j.$ Clearly $t_{ij}t_{jk}t_{ki}=1$ on $U_i \cap U_j \cap U_k$. So the constants $c_{ijk}=\frac{1}{2\pi i}
(\log t_{ij}+\log t_{jk}+ \log t_{ki})$ define an element of $H^2(X\setminus E, \mathbb Z).$ Let $(\chi_i)$ be a partition of unity associated to the cover $(U_i).$ Define on $U_i$ the $\mathcal C^1,$ $1-$ form
$$
\gamma_i= \sum_j \chi_j \frac{dt_{ij}}{t_{ij}}.
$$
  Since $(t_{ij})$ is a multiplicative co-cycle, we have that
$\gamma_i-\gamma_j=\frac{dt_{ij}}{t_{ij}}$ and $(\gamma_i)$ define a continuous $1-$ form on $U_i$, which satisfies $\gamma_i \wedge dt_i=0.$ The 2- form $\Omega= \frac{1}{2\pi i}(d\gamma_j)$ is well defined, closed and continuous in $X \setminus E$ on $U_i$ we have $\Omega \wedge dt_i=0.$

Fix a neighborhood $U$ of $A$ disjoint from $E$, where $\Omega$ is well defined. The cocycle $\{c_{ijk}\}$ restricted to $A$ represents the Chern class of the conormal bundle to $A$ and is hence non trivial. It is also represented by the form $\Omega$ which is hence non exact. So there is a constant $c\neq 0$ such that $\Omega-c\omega=d \gamma$ where $\gamma$ is a $1-$ form. 

Let $[A]$ denote the current of integration on $A.$ Since $\Omega \wedge [A]=0,$ 
we get
  $$<[A],c\omega>= <[A],c\omega-\Omega>=<d[A],\gamma>=0.$$ 
  This is a contradiction.
\end{proof}

\begin{remark}
Let $(X,\mathcal L, E)$ in $\mathbb P^2$ be a lamination which is transversally $\mathcal C^1$ out of E. 
We will see in Section 4, Theorem 15, that if $T$ is a directed closed current for $(X,\mathcal L, E)$, then supp$(T)$ has to intersect $E.$
\end{remark}

\subsection{Baum-Bott Formula}

Suppose $\mathcal F$ is a holomorphic foliation on a compact surface $M.$
If locally $\mathcal F$ is defined on $U_i$, by a holomorphic form $\alpha_i$, then on the intersection $U_i \cap U_j$ we have $\alpha_i=g_{ij}\alpha_j.$ The functions $(g_{ij})$ are the transition functions of the so called normal bundle $\mathcal N_\mathcal F$ of $\mathcal F.$ 

Using the exact sequence

$$
0\rightarrow \mathbb Z \rightarrow^{2i\pi} \mathcal O \rightarrow^{\exp}  \mathcal O^*
\rightarrow 0,
$$
\noindent we get  the long exact sequence of  cohomology groups 

$$
H^1(M,\mathcal O)\rightarrow H^1(M,\mathcal O^*)\rightarrow ^{c_1} H^2(M,\mathbb Z) \rightarrow.
$$

The image of the map $c_1$  of a line bundle $B$ is called the Chern class of $B.$ In the case of $\mathbb P^2$, $H^2(\mathbb P^2,\mathbb Z)=\mathbb Z.$
Hence the Chern class of a bundle is just an integer. For the normal bundle to a foliation of degree $d$ in $\mathbb P^2$ it can be shown, see \cite{Br2000}  or \cite{Z2001} that $N_{\mathcal F}=\mathcal O(d+2).$ So the Chern class is just $(d+2).$ 

We now define the Baum-Bott index at a singular point $p.$ Assume the foliation $\mathcal F$ is given near the singular point $p$ by the form $\alpha=Pdw-Qdz.$ Then $d\alpha=\beta\wedge \alpha,$ with
$\beta=\frac{P_z+Q_w}{|P|^2+|Q|^2} (\overline{P} dz+\overline{Q}dw).$ The Baum-Bott index at $p$ is

$$
BB(\mathcal F,p)=\frac{1}{(2i\pi)^2}\int_{\partial B_1} \beta\wedge d\beta
$$
\noindent where $B_1$ is a small ball around $p.$

So the Baum-Bott index is a residue and it can be shown that it is independent of the form $\alpha$ defining $\mathcal F$ and of the choice of $\beta$ such that $d\alpha=\beta\wedge \alpha.$

When the eigenvalues  $\lambda_1,\lambda_2$ of the vector field $\mathcal F$ at $p$ are both nonzero then

$$
BB(\mathcal F,p)=\frac{\lambda_1}{\lambda_2}+\frac{\lambda_2}{\lambda_1}+2.
$$

\begin{theorem} (Baum-Bott) \cite{BB}
Let $\mathcal F$ be a holomorphic foliation of a compact surface $M.$ Then
$$
N_\mathcal F \cdot N_\mathcal F=\sum_{p \in\; {\mbox{Sing}}(\mathcal F)} BB(\mathcal F,p)
$$ So if $M=\mathbb P^2$, we get 
$$
(d+2)^2= \sum_{p\in \; {\mbox{Sing}}(\mathcal F)}BB(\mathcal F,p).
$$
\end{theorem}
A proof of the previous theorem can also be found in \cite {Br2000} .

Recently, Lins-Neto and Pereira have shown that this is the only relation between the BB indices.
They  consider the space  of holomorphic foliations  of degree $d$ in $\mathbb P^2.$ For each foliation $\mathcal F$ there are $N=d^2+d+1$ singular points $p_1,\dots,p_N$. Define the map
$B(\mathcal F)=[BB(\mathcal F,p_1),
\cdots BB(\mathcal F,p_N)]$ with values in $\frac{(\mathbb P^1)^N}{S_N}$ where $S_N$ is the premutation group of the $N$ coordinates.

\begin{theorem} (Lins-Neto, Pereira)\cite {LNP2006} The Baum-Bott map $B$ has maximal rank equal to $N-1=d^2+d.$
\end{theorem}

\section{Laminations without positive closed currents}

\begin{theorem}
Let $(X,{\mathcal L},E)$ be a lamination with singularities in a compact complex manifold $M$. Either there is a non trivial image
of ${\mathbb C}$ weakly directed by ${\mathcal L}$ or there is a constant $c>0$ such that 
$$
|\phi'(\zeta)|\leq \frac{c}{1-|\zeta|^2}$$
\noindent for all holomorphic maps from the unit disc $\Delta$ into $X,$ weakly directed by ${\mathcal L}.$
The same result holds for maps with images locally contained in leaves outside $E.$
\end{theorem}

\begin{proof}
Assume there is no nontrivial image of $\mathbb C$ weakly directed by $\mathcal L.$ We show that there exists $c>0$ such that $|\psi'(0)|\leq c$ for every holomorphic map $\psi:\Delta \rightarrow X,$ weakly directed by $\mathcal L.$ This implies the result: Let $\psi(z)=\phi(\frac{z+\zeta}{1+\overline{\zeta}z}).$ Then $|\psi'(0)|=|\phi'(\zeta)|(1-|\zeta|^2)\leq c.$
So $|\phi'(\zeta)|\leq \frac{c}{1-|\zeta|^2}.$

\medskip

The proof is by contradiction using Brody's technique. Assume there are $\psi_n$, holomorphic in a neighborhood of $\overline{\Delta}$, such that $|\psi'_n(0)| \rightarrow \infty.$ Let $H_n(z):=(1-|z|^2)|\psi'_n(z)|.$ For each $n,$ there is a $z_n\in \Delta$ so that $H_n(z_n)=\max_\Delta H_n.$ Define
$\phi_n(z):=\psi_n(\frac{z+z_n}{1+\overline{z}_nz})=:\psi_n(E_n(z)).$
Then

\bea
|\phi_n'(z)|(1-|z|^2) & = & |\psi'_n(E_n(z))||E'_n(z)| (1-|z|^2)\\
& = & |\psi'_n(E_n(z))|(1-|E_n(z)|^2)\frac{|E'_n(z)|}{1-|E_n(z)|^2} (1-|z|^2)\\
& = & |\psi'_n(E_n(z))|(1-|E_n(z)|^2)\\
& \leq & H_n(z_n)\\
& = & |\phi_n'(0)|.\\
\eea

Since $|\phi_n'(0)|\geq |\psi_n'(0)|,$ we then have $R_n:=|\phi'_n(0)| \nearrow \infty.$
Define $h_n(z):=\phi_n(\frac{z}{R_n}).$ We have using the above estimate:

$$
|h'_n(z)| = \frac{1}{R_n}|\phi'_n(\frac{z}{R_n})|\leq \frac{1}{1-|z/R_n|^2}.$$
So $\{h_n\}$ is equicontinuous. We can assume that $(h_n)$ converges uniformly on compact sets of ${\mathbb C}$ to a nonconstant holomorphic map $h:\mathbb C \rightarrow X.$ Clearly $h$ is weakly directed by ${\mathcal 
L}$ and the image cannot be contained in $E$, since $\Lambda^2(E)=0.$ The last part is proved in the same way.

\end{proof}

\begin{remark}
It is enough to assume that $E$ is Brody hyperbolic in the sense that there is no holomorphic non constant map $\mathbb C \rightarrow E.$ If $X= E,$ we find the usual Brody Theorem.
\end{remark}

\begin{corollary}
Assume there is a holomorphic map $\phi:\Delta \rightarrow (X,\mathcal L,E)$ such that
$\lim_{r\rightarrow 1} (1-r){\mbox{Area}}(\phi(\Delta_r))=\infty$ and with $\phi(\Delta)$ locally contained in leaves outside of $E.$ Then there is a nonconstant  image of $\mathbb C$ in $(X,\mathcal L,E)$ locally contained in leaves outside $E$ and consequently a positive closed current directed by $(X,\mathcal L, E),$ supported in $\overline{\phi(\Delta)}.$
\end{corollary}

\begin{proof}
If not then $|\phi'(\zeta)|\leq \frac{c}{1-|\zeta|^2}$ and
$\int_{\Delta_r}|\phi'|^2\leq \frac{c'}{1-r},$ a contradiction. The last statement follows from Plante's theorem and the remark after. 

\end{proof}

Brunella \cite{Br2005} has shown that generically for holomorphic foliations by Riemann surfaces in $\mathbb P^k$ there are no directed positive closed currents.

\begin{theorem} (Brunella)
Given $k \geq 2$ and $d \geq 2$, there exists an open dense set $\mathcal{U} \subset \mathcal{F}_d(\mathbb P^k)$ such that any $\mathcal F \in   \mathcal U$ has no non trivial directed positive closed bidimension $(1,1)$ current. A foliation $\mathcal F$ belongs to $\mathcal U$ if all the singularities are hyperbolic and no algebraic curve is invariant by $\mathcal F$.
\end{theorem}

We show that for  $\mathbb P^2,$ if the singularities of $\mathcal F$ are hyperbolic, then the only positive closed currents directed by $\mathcal F$ correspond to algebraic leaves.

\begin{proof}
We show first that there is no mass of a directed positive closed current $T,$ near a singularity. We study the local model near $0,$ given by $\alpha=zdw-\lambda wdz,\lambda=a+ib, b\neq 0.$ Take a transversal $D_{z_0}=\{|w|<1,z=z_0\}.$ The leaves are given by $z=e^\zeta,w=ce^{\lambda \zeta}.$ A leaf hits the transversal at $(z_0,ce^{\lambda \zeta_n})$ with $z_0=e^{\zeta_0}$ and
$\zeta_n=\zeta_0+2i\pi n.$ So $|ce^{\lambda \zeta_n}|=de^{-2\pi nb}$ decays geometrically. The mass of $T$ is finite and is preserved by the holonomy map. This is impossible unless $T$ is concentrated on the separatrices $z=0,w=0.$ Let $\nu(T,p)$ denote the Lelong number of $T$ at $p.$ According to a theorem of Siu \cite{S1974} the sublevel sets $\{p;\nu(T,p)\geq c>0\}$ are analytic. If $T$ has mass on $z=0$ for example, then for some $c>0$, $\{\nu(T,p)\geq c\}$ will define an analytic set of dimension one, this will give an algebraic leaf. We subtract from $T$ these contributions. Then $T$ has to be supported on a closed invariant set $Y,$ without singularities. It can then be shown, Hurder-Mitsumatsu \cite{HM1991}, that
$\int T \wedge T=0.$ This is impossible in the projective space, see also Theorem 17 below.

\end{proof}

We give here the argument that we will use to study the self intersection of a positive directed closed current which gives also a proof of the Hurder-Mitsumatsu result. The result was first proved by Camacho, Lins-Neto, Sad \cite{CLS1988} for holomorphic foliations in $\mathbb P^2$ and by Hurder-Mitsumatsu in the ${\mathcal C}^1$ case.

\begin{theorem}
Let $(X,\mathcal L)$ be a transversally Lipschitz lamination in $\mathbb P^2$ without singularities.Then there is no  positive directed closed current on $X$ except for the current of integration on algebraic leaves.
\end{theorem}

\begin{proof}
Let $T$ be a  positive directed closed current of mass $1$. Let $\Phi_\epsilon$ be an automorphism
of $\mathbb P^2$ close to the identity. Define $T_\epsilon=(\Phi_\epsilon)_*T.$ If in a flow box $B, T=\int [\Delta_\alpha]d\mu(\alpha)$, then $T_\epsilon=\int [\Delta^\epsilon_\beta]d\mu(\beta).$ It is possible to define the geometric intersection $T\wedge_g T_\epsilon.$ If $\phi$ is a test function in $B,$
$$
<T\wedge_gT_\epsilon,\phi>=\int \left(\sum_{p\in J_{\alpha,\beta}^\epsilon}\phi(p)\right)d\mu(\alpha)d\mu(\beta).$$

\noindent Here $J^\epsilon_{\alpha,\beta}$ consists of the intersection points  of $\Delta_\alpha$ and
$\Delta^\epsilon_\beta$. We assume that $\Phi_\epsilon$ is chosen so that there is no common leaf in $\mathcal L$ and $(\Phi_\epsilon)_* \mathcal L.$ We can use a partition of unity to define $T\wedge_g T_\epsilon$ globally. It is shown in \cite{FS2005} that if the lamination is Lipschitz then the number of points in $J^\epsilon_{\alpha,\beta}$ is uniformly bounded.

If $T$ is not the current of integration on an algebraic subvariety the problem can be reduced to the case where the transverse measure is diffuse, i.e. has no mass on points.
We can consider smoothings of $T$ and $T_\epsilon$, $T^\delta$ and $T^{\delta'}_\epsilon.$ Here $T^\delta=\int \phi_*(T) d\rho_\delta(\phi)$ where $\rho_\delta$ is an approximation of identity on $U(3).$ The uniform estimate on the number of intersections between plaques
permits to show that 
$$
<T^{\delta'}_\epsilon, \phi T^\delta> \rightarrow 0
$$
\noindent when $\delta,\delta'<<\epsilon, \epsilon \rightarrow 0.$ 
On the other hand we can show that

$$
\int T \wedge T =\int T^{\delta'}_\epsilon \wedge T^\delta =\int T^{\delta'}_\epsilon \wedge_g T^\delta \rightarrow 0.
$$

This is impossible for positive closed currents in $\mathbb P^2$ unless $T=0.$Indeed the integral can be computed cohomologically and T is cohomologous to the Fubini Study form.

\end{proof}

Let $K$ be a laminated set. We define a holonomy map. Fix a point $p\in K$. Suppose there is a closed curve $\gamma$ in the leaf through $p$ and which starts and ends at $p.$ Let $A$ denote a transverse slice to $K$ at $p.$ Then for every $q\in A\cap K$ close enough to $p$ we can follow the leaf through
$q$ along $\gamma$ until we return to a point $s(q)\in A \cap K.$ The map $s$ is said to be a holonomy map. If the lamination is a holomorphic foliation, and $A$ is a complex manifold, then $s$ is a holomorphic map. We note that if  there is a closed laminated current on $K$, then the holonomy map preserves the transverse measure. If $K$ is in a complex surface and $s$ has $p$ as an attracting fixed point, then
there can be no mass of a transverse measure on small punctured discs in $A.$ Bonatti, Langevin and Moussu \cite{BLM1992} have proved the following result:

\begin{theorem}(Bonatti-Langevin-Moussu)
Let $\mathcal F$ be a holomorphic foliation of $\mathbb P^2.$ Assume $M$ is an exceptional minimal set for $\mathcal F$. Then there is a leaf $L$ in $M$ and a closed real curve in $L$ with hyperbolic holonomy.
\end{theorem}

\subsection{Poincar\'e Metric}

Let $(X,\mathcal L,E)$ be a laminated set with singularities in a compact Hermitian manifold $(M,g).$

\medskip

We define two functions for $p\in X$ with values in $[0,\infty].$  
$$
\eta(p)=\max \{|\phi'(0)|,\phi:\Delta \rightarrow X,\phi(\Delta)\;
{\mbox{ locally in leaves outside}}\; E,\phi(0)=p\}.
$$
$$
\tilde{\eta}(p)=\max\{|\phi'(0)|, \phi:\Delta \rightarrow X \setminus E
, \phi(\Delta)  \; {\mbox{contained in a leaf and}} \; \phi(0)=p,$$
$${} if  \; p \in E,\;
{\tilde{\eta}(p)=0}\}.
$$

The following is a variation on results by Ghys \cite{Gh1999}, Candel-Gomez Mont \cite{CG-M1995}, Verjovsky \cite{V1987} and Candel \cite{C1993}. We use a result by Royden \cite{R1974} in order to study the lower
semi continuity of the above functions.

\begin{theorem}
 If there is no image of $\mathbb C$ locally contained in leaves outside $E$, then $\eta$ is  upper semicontinuous on $X\setminus E$ with values in $(0,\infty)$.
\end{theorem}

\begin{proof}

The maps under consideration are allowed to pass through $E.$ We show 
that $\eta$ is upper semicontinuous. Suppose that $p_n \rightarrow p$ and $|\phi'_n(0)|>\eta(p_n)-1/n, \phi_n(0)=p_n.$
By Theorem 3 there is a constant $c$ such that $|\phi'_n(\zeta)|\leq \frac{c}{1-|\zeta|}.$ We can then assume that $\phi_n \rightarrow \phi$ uniformly on compact sets of $\Delta.$ Then $\phi(0)=p,
|\phi'(0)|\geq\overline{\lim}\;\eta(p_n).$ Since $\phi$ is an element in the family on which the max is taken, it follows that $\eta(p)\geq \overline{\lim}_{p'\rightarrow p} \eta(p')$. So $\eta$ is upper semi continuous..

\end{proof}

\begin{theorem}
The function $\tilde{\eta}$ is lower semicontinuous. Assume that $E$ is finite and that for $p\in E$ there are exactly $k$ local invariant hypersurfaces $\Sigma_1,\dots,\Sigma_k$ such that $\cap \Sigma_i=p.$ 
The function $\tilde{\eta}$  is continuous if there is no image of $\mathbb C$ locally contained in leaves out of $E.$
\end{theorem}

\begin{proof}
We first prove lower semicontinuity. Pick a point $p\in X.$ If $p \in E,$ $\tilde{\eta}(p)=0$ so lower semicontinuity is obvious. We assume that $p\in X \setminus E$ and $\tilde{\eta}(p)<\infty.$ Let $L\subset X \setminus E$ be the leaf through $p$ and let $\phi:\Delta \rightarrow L$ be the uniformizing map, $\phi(0)=p.$ Fix $r<1.$ Then $\phi(\overline{\Delta}_r)$ admits a Stein neighborhood $U$, (Royden \cite{R1974}), biholomorphic to the bidisc. If $\phi$ is not $1-1$ we lift to a Riemann domain. Steiness implies the existence of holomorphic retractions on submanifolds. So the leaves are holomorphic graphs on $\phi(\overline{\Delta}_r)$ and we can lift $\phi$ to the graph. So at $p'$ close to $p$,
$\tilde{\eta}(p') \geq \tilde{\eta}(p)-\epsilon.$ This gives lower semicontinuity. 
\medskip

 We prove next upper  semicontinuity under the hypothesis that there is no image of $\mathbb C$ locally contained in leaves outside of $E.$ Let $p\in X, p_n \rightarrow p,
|\phi_n'(0)|\geq \tilde{\eta}(p_n)-1/n, \phi_n(0)=p_n.$
We can assume $\phi_n\rightarrow \phi$. If $\phi(\Delta) \cap E =\emptyset,$ then $\tilde{\eta}(p) \geq \limsup \tilde{\eta}(p_n)$.

Assume next that $\phi(\Delta) \cap E$ contains some $p_0.$ 
We want to show that $\phi(\Delta)=\{p_0\}.$ Let $q\in D \subset \Delta$ be a small subdisc with $\phi(q)=p_0.$
If $\phi(\Delta) \neq \{p_0\}$, for some $i,$ $\phi(D)$ is not locally contained in $ \Sigma_i.$ Assume
$\Sigma_i=\{h_i=0\}$ in a neighborhood of $p_0.$ For large $n,$ since $\phi_n$ are contained in leaves, $h_i\neq 0$ on $\phi_n(D)$. But $h_i\circ \phi_n(q) \rightarrow 0.$ By Hurwitz,
$\phi(D) \subset \{h_i=0\},$ a contradiction.
Hence $\phi$ is constant so $p_0=p$ and $p_n \rightarrow p_0,$ 
$|\phi'_n(p)| \rightarrow 0$ and
$\tilde{\eta}(p_n) \rightarrow 0.$

\end{proof}

\begin{example}
\cite{CG-M1995} Suppose that $\mathcal L$ is given locally by holomorphic vector fields with linearizable hyperbolic singularities i.e. equivalent to
$Z=\sum_{j=1}^k \lambda_j z_j \frac{\partial}{\partial z_j}, \lambda_j \neq 0.$ Then, the first hypothesis of Theorem 20 is satisfied.  If we assume that there are no algebraic leaves, then Brunella's result, Theorem 16, implies that there are not directed closed currents and hence no directed image of $\mathbb C$ .
Then  Theorem 20 applies.
\end{example}

\begin{remark}
Suppose that we are in the situation of the second part of Theorem 20.
Since the leaves in $X \setminus E$ are uniformized by the unit disc we have a Poincar\'e metric $g_P$ on each leaf. There is also a metric $g_L$ induced by the Hermitian metric $g$ on $M.$   Then $g_L=\hat{\eta} g_P,$ and $\hat{\eta}=\tilde{\eta}.$ So the previous result gives the continuity of $\hat{\eta}$ in $X \setminus E.$
\end{remark}

Let $\mathcal C_d$ denote the set of foliations $\mathcal F$ of degree $d$ in $\mathbb P^k,$ which admit a non constant directed holomorphic map $f:\mathbb C \rightarrow \mathbb P^k$. Using Brody's technique it is easy to check  that $\mathcal C_d$ is closed in $\mathcal F_d(\mathbb P^k).$
Observe also that Corollary 1 implies that $\mathcal C_d$ is contained in the set $C_d$ of foliations admitting a positive closed directed current. Theorem 7 and Theorem 16 imply that $\mathcal U_d:=\mathcal F_d(\mathbb P^k)\setminus \mathcal C_d$ contains a dense real Zariski open set of $\mathcal F_d(\mathbb P^k)$. We can define on $\mathcal F_d \times \mathbb P^k,$ the functions $\eta(\mathcal F,p)$,  $\tilde{\eta}(\mathcal F,p)$ as the functions $\eta$, $\tilde\eta$ associated to $\mathcal F$, we just emphasize the dependence on $\mathcal F.$ Theorem 20 extends easily to this context and the function $(\mathcal F,p) \rightarrow \tilde{\eta}(\mathcal F,p)$ is lower semi continuous on $\mathcal F_d \times \mathbb P^k$ and continuous on $\mathcal U_d.$ The function

$$
c(\mathcal F):=\sup\{|\phi'(0)|; \phi:\Delta \rightarrow \mathbb P^k,
\; {\mbox{holomorphic and }}\; \mathcal F\; {\mbox{directed}}\}=\sup_p \tilde{\eta}(\mathcal F,p),
$$
\noindent is also continuous on $\mathcal U_d$ and is an exhaustion in $\mathcal U_d,$ i.e. it goes to infinity when $\mathcal F$ approches the boundary of  $\mathcal U_d$ which is $\mathcal C_d,$ since this set has empty interior.

\subsection{Universal covering maps of leaves.}

The following theorem is proved for laminations in a compact complex manifold $M$ of dimension $k$.
 
\begin{theorem}
Let $(X,\mathcal L, E)$ be a laminated set with a finite set $E$ of singular points. Assume there is no directed positive closed current on $(X,\mathcal L,E).$ Let $\phi:\Delta \rightarrow X \setminus E$ be the universal covering map of a leaf $L.$ Then 
$$
\int_\Delta (1-|\zeta|)|\phi'(\zeta)|^2d\lambda(\zeta)=\infty.
$$
\end{theorem}

\begin{proof}
Recall that $|\phi'(\zeta)|\leq \frac{c}{1-|\zeta|},$ with $c$ independent of $\phi.$ We need an estimate from below.

\begin{lemma}
For every $\epsilon>0$ there is a constant $c_\epsilon>0$ such that 
$$
|\phi'(\zeta)|\geq \frac{c_\epsilon}{1-|\zeta|}
$$
\noindent if dist$(\phi(\zeta),E)\geq \epsilon.$
\end{lemma}

\begin{proof}
We only have to show that $|\phi'(0)|\geq c_\epsilon$ if dist$(\phi(0),E) \geq \epsilon.$
We then consider $\psi(\zeta)=\phi(\frac{\zeta+\zeta_0}{1+\overline{\zeta}_0\zeta}).$ Then
$\psi'(0)=\phi'(\zeta_0)(1-|\zeta_0|^2)$ and we get the estimate on $\phi'(\zeta_0).$

\medskip

Let $(B_i)$ be a locally finite covering of $X \setminus E$, by flow boxes. We can assume that the plaques in $B_i$ are graphs $w_i=f_\alpha(z_i)$, $|z_i|<1,$ and that the graphs extend uniformly to $|z_i|<2.$ Let $\phi:\Delta \rightarrow L \subset X \setminus E$ be the universal covering of a leaf $L.$ 
We can assume that $\phi(0)=(a,b)$ is on the graph $w_i=f_\alpha(z_i)$. Then $\phi^{-1}$ has a well defined branch, $\phi^{-1}(\phi(0))=0,$ on $w_i=f_\alpha (z_i), |z_i-a|<1.$ By the Schwarz's Lemma, $|(\phi^{-1})'|\leq 1$. Hence $|\phi'(0)| \geq 1$.

The constant $c_\epsilon$ appears when we compare the metric in normalized coordinates with the Hermitian metric on $M.$

\end{proof}

We continue with the proof of the Theorem.
Consider the measurable subset $A$ of the unit circle defined as 
$$
A:=\{e^{it}; \lim_{r \rightarrow 1}\phi(re^{it})\in E\}.$$ Let $\lambda_1$ denote the Lebesgue measure on the unit circle. We want to show that $\lambda_1(A)=0.$ Assume that $\lambda_1(A)>0.$     We cover $E$ by finitely many small domains $U_i$ in which $E$ is defined by holomorphic functions $h_{i,j}.$ We can assume that no $h_{i,j}$ vanishes identically on any open piece of $\phi(\Delta).$  Let $A_i=\{e^{i\theta}; \lim \phi(re^{i\theta})\in E \cap U_i\}.$ Suppose that $\lambda_1(A_1)>0.$
We have that $\lim_{r \rightarrow 1}h_{1,j}(\phi(re^{i\theta}))=0$ for all $\theta\in A_1$. The argument in Privalov \cite{CL1996} shows that we can construct a domain with rectifiable boundary such that $h_{1,j}(\phi)$ is well defined there, with radial limits $0$ on a set of positive measure. The estimate $|\phi'(\zeta)|\leq \frac{c}{1-|\zeta|}$  implies that the function  $h_{1,j}(\phi)$ is bounded in an angle.  By Lindel\"{o}f's Theorem, $\phi$ has non tangential limit in  a possibly smaller angle. Privalov's theorem implies that $h_{1,j}\circ \phi \equiv 0,$ a contradiction.
 Consequently there is a $\delta>0$ and a closed set $F \subset \partial \Delta$ with $\lambda_1(F)\geq 1/2,$ such that for every $e^{it}\in F$ there is a sequence $r_j \rightarrow 1$ with dist$(\phi(r_je^{it}),E))\geq 2\delta.$ Define $E_\delta:=\{q;{\mbox{dist}}(q,E)\leq \delta\}.$ Since $|\phi'(\zeta)|\leq \frac{c}{1-|\zeta|},$ then $\phi(\Delta(r_je^{it},a(1-r_j))$ does not intersect $E_\delta$ for some $a>
 0$ small enough. By Lemma 1, on the disc $\Delta(r_je^{it},a(1-r_j))$ we have $|\phi'(\zeta)| \geq \frac{\delta}{1-|\zeta|}.$ Consequently

$$
\int_0^1|\phi'(re^{it})|^2(1-r)dr \geq \sum_j \int_{r_j-a(1-r_j)}^{r_j+a(1-r_j)}\frac{c_\delta dr}{1-r}
=c_\delta \sum_j \log \frac{1+a}{1-a}.$$

Hence $\int_Fdt \int_0^1 |\phi'(re^{it})|^2(1-r)dr=\infty.$

\end{proof}

\section{Harmonic Currents}

As we have seen in the previous chapter, generic holomorphic foliations of $\mathbb P^k$ with hyperbolic singularities have no algebraic leaves and hence don't admit a directed $\underline{\mbox{closed}}$ current. It is of interest to find another global object describing how leaves are distributed. Harmonic currents provide such a description.

\medskip

The notion of harmonic measure was introduced by L. Garnett \cite{G1983} for foliations without singularities in Riemannian manifolds. It was extended to laminations with singularities by Berndtsson and the second author \cite{BS2002} and developed in \cite{FS2005}. 

\medskip

Let $(M,\omega)$ be a compact K\"{a}hler manifold. Recall that a $(p,q)$ current $T$ is harmonic if and only if $i\partial \overline{\partial} T=0.$ We will be mostly interested in positive harmonic currents of bidimension $(1,1)$ or of bidegree $(1,1)$. For laminations they will be the analogue of positive invariant measures for dynamical systems.

\medskip

Let $(X,\mathcal L,E)$ be a lamination by Riemann surfaces in $M$ with $\Lambda^2(E)=0.$  Assume ${\mathcal L}$ is given by continuous $(1,0)$ forms $(\gamma_j)$ vanishing on $E.$ We will also assume that $E$ is locally complete pluripolar. This means that for $p\in E$, there is a plurisubharmonic function $u$ in some neighborhood $B(p,r)$ such that $(u=-\infty)=E \cap B(p,r).$

\subsection{Construction of harmonic currents by exhaustion}

\begin{theorem} 
Let M be a compact Hermitian manifold.
Let $\phi:\Delta\rightarrow M$ be a holomorphic map. Assume 
$$
\int (1-|\zeta|)|\phi'(\zeta)|^2d\lambda(\zeta) = \infty.
$$
Define $T_r:=\phi_*(\log^+ \frac{r}{|\zeta|} [\Delta])$ and $\tau_r:=\frac{T_r}{\|T_r\|}.$ Then all cluster points of $\{\tau_r\}$ are positive harmonic currents supported on $\overline{\phi(\Delta)}.$ If $\phi$ is weakly $\mathcal L-$directed for $(X,\mathcal L,E)$\\ (resp. $\mathcal L-$directed ) then all the cluster points of $\tau_r:$ are weakly $\mathcal L-$directed (resp. $\mathcal L-$directed ). 
\end{theorem}

\begin{proof}

The situation is different from the construction of positive closed currents starting from an image of $\mathbb C.$ Then we have to choose a cluster point, not all of them are a priori closed. Here all cluster points are harmonic.

Define in $\mathbb C$, the function $G_r (\zeta)=1/(2\pi)
\log^+\frac{r}{|\zeta|},r>0.$  Observe that $G_r$ vanishes on $\partial \Delta_r$
and that near $0$ the Laplacian of $G_r$ is the Dirac mass at $0.$ Let $\phi(0)=p.$ Define $T_r=\phi_*(G_r[\Delta]),r<1.$ If $\theta$ is a test form, then
$$
<T_r,\theta>=1/(2\pi)\int \log^+(r/|\zeta|)\phi^*(\theta).
$$
We can estimate the mass of $T_r.$ We have

\bea
A(r):=\|T_r\| & = & \int \log^+(r/|\zeta|) |\phi'(\zeta)|^2 d\lambda(\zeta)\\
& \sim & \int (r-|\zeta|)^+ |\phi'(\zeta)|^2 d\lambda(\zeta)\\
&  & \\
\eea

The hypothesis implies that $A(r) \nearrow \infty$. We also have from Jensen's formula 
$$
i\partial \overline{\partial} T_r =\phi_*(\nu_r)-\delta_p$$
\noindent where $\nu_r$ denotes the Lebesgue measure on the circle of radius $r$. Hence if $T$ is a cluster point of $(T_r/A(r))$, $T$ is non zero, and

\bea
i\partial \overline{\partial} T & = & \lim_{r\rightarrow 1}
\frac{i\partial \overline{\partial}T_r}{A(r)}\\
& = & \lim_{r\rightarrow 1} \frac{\phi_*(\nu_r)-\delta_p}{A(r)}\\
& = & 0\\
\eea

Clearly if $\phi(\Delta)$ is weakly directed by continuous forms $\gamma_j,$ then
$T \wedge \gamma_j=0.$

We show that if $\phi$ is $\mathcal L-$ directed for $(X,\mathcal L,E)$ i.e. if  $\phi(\Delta)$  are locally contained in leaves then all the cluster points are harmonic  $\mathcal L-$directed. In a flow box the current $T_r$ is supported on a finite number of plaques. The fraction of plaques that is only partially covered after shrinking the plaques is of bounded mass, since $A(r)$ is unbounded we can neglect that contribution in the limit. We then observe that the limit, in a flow box of directed positive harmonic currents, is directed. Indeed the limit is in the closed convex hull of the extreme points, which gives  the decomposition in plaques.

\end{proof}

\begin{corollary}
Let $(X,\mathcal L,E)$ be a lamination with $E$ analytic of dimension $0$ and without directed positive  closed currents. Then the leaves are uniformized by discs $\phi:\Delta \rightarrow (X,\mathcal L)$ and the cluster points of $T_r = \frac{1}{A(r)}\phi_*(G_r[\Delta])$ are harmonic currents directed by $\mathcal L.$ 
\end{corollary}

\begin{proof}
This is a consequence of Theorem 22 and of the estimate in Theorem 21.

\end{proof}
\medskip
The currents $T_r$ give an average along leaves and our goal is to prove an ergodic theorem: the limit  of the averages, is independent of the leaf.

We now give another construction of harmonic currents for general singular sets $E$ as above, when the hypothesis in Theorem 22 is not necessarily satisfied.

\begin{theorem}
 Assume the lamination $(X,\mathcal L,E)$ is transversally ${\mathcal C}^2.$ There exists a positive current $T$ of bidimension $(1,1)$ supported on $X$ such that\\
\noindent i) $\int T \wedge \omega=1.$\\
\noindent ii) $T \wedge \gamma_j=0\; \forall \; j.$\\
\noindent iii) $i \partial \overline{\partial} T=0.$\\
\noindent iv) In a flow box $B,$ disjoint from the singularities
$$
T=\int_\Delta h_\alpha [V_\alpha]d\mu(\alpha)
$$
\noindent where $[V_\alpha]$ denotes the current of integration on the plaque $V_\alpha$, $h_\alpha$
is a positive harmonic function depending measurably on $\alpha$, when $\alpha$  varies on a transversal $\Delta.$
\end{theorem}

We show how to find a current satisfying i)-iii). For iv) we refer the reader to \cite{FS2005}. The equivalence of weakly directed and directed
harmonic currents is proved in \cite{FWW2007b} for surfaces without the assumption of transverse smoothness.


\begin{proof}
Let $$C:=\{T \geq 0, \int T\wedge \omega=1, T \wedge \gamma_j=0,
T\; {\mbox{of bidimension}}\; (1,1),\;{\mbox{supported on}}\; X\}.$$
Define $F=\{i\partial \overline{\partial} \psi\}^\perp,$ the orthogonal complement in the space of currents of bidimension $(1,1)$ of $\partial \overline{\partial}-$ exact $(1,1)$ forms,The space F is the space of $\i\partial \overline{\partial} $ closed currents. We have to show that $C \cap F \neq \emptyset.$ Assume not, then by the Hahn-Banach theorem and the reflexivity of the space of currents, there is a $\delta>0$ and a smooth function $\psi$ such that for every $T$ in $C$,
$$
<T,i\partial \overline{\partial} \psi>\geq \delta.\;\;\;\;\;\;\;\;(1)
$$

Since $\psi$ is smooth, by $(1)$ it is strictly subharmonic along the leaves. The space $X$ being compact, the maximum of $\psi$ can be reached only at a point $p\in E.$ Choose local coordinates $z$ in $M$ such that $p=0$ and $E \cap B(0,r)=\{u=-\infty\},$ with $u$ plurisubharmonic in $B(0,2r)$ and negative. The function $\psi_1:=\psi-\delta\|z\|^2/2$ is still plurisubharmonic along leaves and
$\psi_1(0)=\max_{X\cap \overline{B(0,r)}}\psi_1.$ We have
$\psi_1(z) \leq \psi_1(0)-\delta r^2/2$ on $\partial B(0,r) \cap X.$ Define $\psi_2=\psi_1+\epsilon u.$
Then ${\psi_2}_{| \partial B_r\cap X}\leq \psi_1(0)-\delta r^2/2.$ If $z_1\in X \setminus E, u(z_1)\neq -\infty$ is close enough to zero and $\epsilon$ is small enough we get $\psi_2(z_1)>\psi_1(0)-\delta r^2/2$.
Hence the max of $\psi_2$ on $X\cap \overline{B(0,r)}$ is reached away from $\partial B(0,r)$ and away from $E$. This contradicts that $\psi_2$ is strictly subharmonic on leaves. Since $\Lambda^2(E)=0,$ $T$ has no mass on $E$ as follows easily from the existence of Lelong numbers for harmonic currents
\cite{BS2002}.

In a flow box $B$, we can  assume that the lamination is given by the kernel of continuous $(1,0)$ forms $(\gamma_j).$ We consider the desintegration of $\|T\|$ along leaves.

$$
<T,\phi>=\int <\nu_\alpha,\phi>d\mu(\alpha).
$$
Here $
 \nu_\alpha 
$is a measure on the plaque $V_\alpha.$.
We extend $T$ to forms which are ${\mathcal C}^2$ along leaves and continuous and  then the extension is still $\partial \overline{\partial}$ closed along leaves, \cite{FS2005}.  For $\mu$ almost every $\alpha, \nu_\alpha$ is $i\partial\overline{\partial}$ closed along leaves and hence $ 
 \nu_\alpha =
h_\alpha[V_\alpha] $ is a positive harmonic function on the corresponding plaque.

When the lamination is transversally $\mathcal C^2$ it is possible to show, as in Theorem 10, that such currents can be decomposed as claimed, by considering the decomposition in extremal elements. For general $\mathcal L$ one has to choose an appropriate $T$ \cite{FS2005}.

\end{proof}

We give a result on decomposition of directed harmonic currents.

\begin{proposition}
Let $\mathcal F$ be a holomorphic foliation with singularities on a compact surface $M$ without meromorphic first integral. Let $T$ be a positive harmonic current directed by $\mathcal F.$ Then
$$
T=\sum_{j=1}^N \lambda_j[V_j]+T_0
$$
\noindent with $\lambda_j>0$, $[V_j]$ are closed analytic subvarieties directed by $\mathcal F$ and
$T_0$ is diffuse in the following sense. In a flow box $B,$ $T_0=\int h_\alpha [V_\alpha]d\mu(\alpha),$
with the measure $\mu$ diffuse.
\end{proposition}

\begin{proof}
It is shown in \cite{FS2005}, Proposition 5.1 that if a positive directed harmonic current gives mass to a leaf then the leaf is contained in a compact Riemann surface. Theorem 6.1 implies that the number of such leaves is finite. The proposition follows.
\end{proof}

\begin{remark}
\noindent i) If $\Lambda^2(E)=0$ or if $E$ does not contain a closed set $E'$ such that $\mathbb P^2\setminus E'$ is $k-1$ pseudoconvex, then $E$ does not support a nonzero positive harmonic $(1,1)$ current \cite{FS1995}. 

\noindent ii) The measure $\mu$ is not necessarily invariant by holonomy, but if $A$ is a Borel set on a transversal, such that $\mu(A)=0$, then $\mu(\gamma(A))=0$ if $\gamma$ is in the pseudogroup of holonomy transformations. Indeed Harnack's inequalities show that if $T_a$ denotes the slice of the current $T,$ if $\phi (a,a')$ denotes the holonomy map from the a slice to the a' slice  then the direct image  $(\phi(a,a'))_* T_a \leq cT_{a'}$ with $c$ the constant given by Harnack's inequalities.
\end{remark}

\begin{corollary}
Let $(X,\mathcal L,E)$ be a lamination with singularities in $\mathbb P^2$. Assume that $\Lambda^2(E)=0.$ Let $L_1,L_2$ be two leaves of $\mathcal L.$
Then $\overline{L}_1 \cap \overline{L}_2\neq \emptyset.$
\end{corollary}

\begin{proof}
If not there are two positive closed or harmonic currents $T_1,T_2$ supported respectively on $\overline{L}_1, \overline{L}_2$, this is basically a consequence of the maximum principle, see\cite{BS2002}. It is shown in \cite{FS1995}
that the complement of the support of such currents is pseudoconvex hence, it is Stein. However, $\mathbb P^2\setminus {\mbox{supp}}(T_1)$ which is Stein, cannot contain a compact (Supp($T_2$)) with Stein complement. This follows from the existence of a strictly p.s.h exhaustion function in $\mathbb P^2\setminus {\mbox{supp}}(T_1)$ .

\end{proof}

For each $d \geq 2,$ Loray-Rebelo \cite{LR2003} constructed a non empty open set ${\mathcal U}$ of holomorphic foliations of degree $d$ in $\mathbb P^k$ by Riemann surfaces such that every leaf of ${\mathcal F}\in {\mathcal U}$ is dense. By Theorems 7, 10 and 16 we can assume that $\mathcal F \in \mathcal U$ has isolated hyperbolic singularities and that $\mathcal F$ supports no weakly directed or directed positive closed current. The (weakly) directed harmonic currents $T_{\mathcal F}$ associated to such foliations have support equal to $\mathbb P^k.$ 

\subsection{Potential Theory of positive harmonic currents and energy.}

The theory can be developed in compact K\"{a}hler manifolds \cite{FS2005}..
For simplicity we restrict to $(\mathbb P^k,\omega)$ where $\omega$ is the Fubini Study form normalized by $\int \omega^k=1.$

\begin{theorem}
Let $T$ be a positive, $\partial \overline{\partial}$ closed current of bidegree $(1,1)$ on $\mathbb P^k.$ Then the current can be decomposed as
$$
T=c\omega+\partial S+\overline{\partial S}+i\partial \overline{\partial} u.
$$
\noindent with $c\geq 0$, $S$ a $(0,1)$ form in $L^2$, with $\partial S$ and
$\overline{\partial} S$ in $L^2$ and $u$ is a function in $L^1.$ The current $\overline{\partial} S$ is uniquely determined by $T,$ moreover $T$ is closed if and only if $\overline{\partial}S=0.$ The map which sends T to $\overline{\partial} S$ is well defined.

\end{theorem}

\begin{proof}
The current $\overline{\partial}T$ is closed and $\overline{\partial} $ exact,
hence the $\partial \overline{\partial}$ lemma \cite{De 1996}, implies the existence of a $(0,1)$ current $R$, such that 
$$
\overline{\partial}T=\overline{\partial}\partial R.
$$
Hence $\overline{\partial}[T-\partial R]=0.$ Therefore there is a constant
$c \geq 0$ and a current $R'$ such that

$$
T-\partial R = c\omega+\overline{\partial} R'.
$$

Using that $T = \overline{T}$ we get the decomposition
$$
T=c\omega+\partial S_0+\overline{\partial S_0}\;\;\;\;\;\;\;(1)
$$
\noindent where $S_0$ is a $(0,1)$ current.

\medskip

Observe that $S_0$ is not unique, it is defined up to addition of $\overline{\partial}v.$ Hence $\overline{\partial}S_0$ is unique.

We show next that
$$
E(T) := \int \overline{\partial} S_0 \wedge \partial {\overline{S}}_0\wedge \omega^{k-2}<\infty.
$$

Assume at first that $T$ is smooth. Then $S_0$ can be chosen smooth and
\bea
0 & \leq & \int T \wedge T \wedge \omega^{k-2}\\
& = & c^2 \int \omega^k+2\int \partial S_0\wedge \overline{\partial S_0}\wedge\omega^{k-2}\\
& = & \left| \int \omega^{k-1}\wedge T\right|^2-2\int \overline{\partial}S_0\wedge \partial \overline{S}_0\wedge \omega^{k-2}\\
\eea

So we get the a priori estimate 
$$
2\int \overline{\partial} S_0\wedge \partial \overline{S}_0 \wedge \omega^{k-2}
\leq \left| \int T \wedge \omega^{k-1}\right|^2. \;\;\;\;\;(2)
$$

For an arbitrary $T,$ we can approximate $T$ by smooth $(T_\epsilon)$. By (2)
we get a uniform bound for the norm $\overline{\partial}S_0^\epsilon$
in $L^2.$ So letting $\epsilon \rightarrow 0,$ we get that $\overline{\partial}S_0
\in L^2.$ By Hodge $L^2$ theory on $\mathbb P^k,$ \cite{De 1996} we can solve 
$$
\overline{\partial}S=\overline{\partial}S_0
$$
\noindent with $S,\partial S$ in $L^2$ and $S_0=S+\overline{\partial} v.$
Replacing in (1) we obtain that
$$
T=c\omega+\partial S+\overline{\partial S}+i\partial \overline{\partial}\left(
\frac{v-\overline{v}}{i}\right)$$
\noindent and the function $u=(v-\overline{v})/i$ is in $L^p,$ for any $p<\frac{k}{k-1}.$

\end{proof}

\begin{corollary}
Let $T$ be a positive harmonic current of bidegree $(1,1)$ in $\mathbb P^k$.
Let $X=$Supp$(dT)$. Then $\mathbb P^k \setminus X$ is pseudoconvex.
If $T$ is not closed, then $X$ is non polar for the Sobolev space $W^{1,2}$.
In particular $X$ is locally non pluripolar and $X$ is not of finite $\Lambda^{2k-2}$ Hausdorff measure.
\end{corollary}

\begin{proof}
If  $T$ is closed, then X is empty and $\mathbb P^k $ is pseudoconvex.Hence we can assume that $T$ is not closed. From Theorem 24, $-\partial T= \overline{\partial} (\partial \overline{S}),$ and
$\partial \overline{S}$ is a $(2,0)$ form in $L^2.$ Since $dT$ is supported on $X$, $\partial T$ is also supported on $X$, and hence $\overline{\partial}(\partial \overline{S})$ as well. 
It follows that $\partial \overline{S}$ is holomorphic 
in the complement of $X$ and does not extend  holomorphically across any point in the boundary of X. Hence, $\mathbb P^2\setminus X$ is pseudoconvex.

If $X$ is polar in $W^{1,2}$, in particular if it is pluripolar,  there is a sequence
$(\chi_n)$ of smooth functions $0 \leq \chi_n\leq 1$
vanishing near $X$, $\chi_n \rightarrow 1_{\mathbb P^k \setminus X}$ pointwise
and such that $\int \partial \chi_n \wedge \overline{\partial} \chi_n \wedge \omega^{k-1}
\leq C<\infty.$ 
We want to show that $f=\widetilde{\partial \overline{S}}$, the extension of $\partial \overline{S}$
by zero through $X$ is holomorphic, hence $T$ is closed.
We have
$\chi_n \partial \overline{S} \rightarrow f$ in $L^2$ and
$\overline{\partial}(\chi_n \partial \overline{S})=\overline{\partial}\chi_n \wedge \partial \overline{S}.$

Fix $\epsilon>0$, choose $U_\epsilon$ a neighborhood of $X$ such that
$\int_{U_\epsilon} \partial \overline{S}\wedge \overline{\partial}S \wedge \omega^{k-2}\leq \epsilon,$ this is possible since $X$ is of measure zero.
Clearly $\overline{\partial} \chi_n \wedge \partial\overline{S} \rightarrow 0$ on compact subsets of $\mathbb P^k \setminus X.$ Let $\psi$ be a smooth $(1,0)$ form. We have by Schwarz's inequality

\bea
\left|\int_{U_\epsilon}\overline{\partial} \chi_n\wedge  \partial \overline{S} \wedge \psi \wedge \omega^{k-2}
\right| & \leq & 
\left|\int_{U_\epsilon}\partial{\chi_n} \wedge \overline{\partial} \chi_n  \wedge \psi 
\wedge{\overline{\psi}}\wedge \omega^{k-2}\right|^{1/2}\\
& \cdot & 
\left|\int_{U_\epsilon} \partial \overline{S}\wedge \overline{\partial}S \wedge \omega^{k-2}
\right|^{1/2}\\
& \leq & C(\psi) \epsilon^{1/2}.
\eea

Since $\psi$ and $\epsilon$ are arbitrary, $f$ is $\overline{\partial}$ closed. Hence $f=0$ and $T$ is closed.
When  $X$ has finite $\Lambda^{2k-2}$ Hausdorff measure, we can cover X by cubes and use standard bump functions to show that $\int i\partial \chi_n \wedge \overline{\partial} \chi_n \wedge \omega^{k-1}
\leq C<\infty.$

\end{proof}

\begin{remark}

\noindent i) If $Y$ is an invariant set (possibly with singularities) in $\mathbb P^2$ not supporting a positive closed 
directed current then $\Lambda^2(Y)=\infty$ and $Y$ is not polar. It is observed in \cite{Br2005} that   
 Thm.24  implies that a minimal exceptional se in $ {\mathbb P^2}$ is not pluripolar.
\end{remark}

\subsection{Harmonic Currents of finite energy}
We are going to introduce a real Hilbert space,of infinite dimension, of classes of harmonic currents,the elements will be a space of $\partial \overline{\partial}$ cohomology classes. It is the analogue of the Hodge cohomology space for closed currents, a finite dimensional space.
Let $T$ be a real current of order $0$ and bidegree $(1,1)$ in $\mathbb P^k.$
If $T$ is $\partial \overline{\partial}$ closed it can be decomposed as in the previous paragraph as
$$
T=c\omega+\partial S_0+\overline{\partial S}_0,
$$

\noindent where $c\in \mathbb R$ and $S_0$ is a $(0,1)$ current. The current $\overline{\partial}S_0$ is also uniquely determined by $T.$ But a priori $\overline{\partial}S_0$ is not necessarily in $L^2.$ We say that $T$ is of finite energy $E(T)$
if and only if 
$$
E(T):=<\overline{\partial}S_0 \wedge \partial \overline{S}_0 \wedge \omega^{k-2}><\infty.$$
We define the Hilbert space $H_e$ of classes of real currents $T$ as the space of classes $[T]$ of bidegree $(1,1)$ of order $0$ with the norm
$$
\|T\|^2_e:=|<T,\omega^{k-1}>|^2+E(T)\;\;\;\;\;(3).
$$
Observe that if $\|T\|_e=0$ then $c=0$ and $\overline{\partial}S_0=0$, hence $\overline{\partial}T=-\partial \overline{\partial}S_0=0$ therefore $T$ is closed and orthogonal to $\omega$, so $T=i\partial \overline{\partial}u$ for some $u$ in $L^1.$ The converse is clear: If $T=i\partial \overline{\partial}u$ then $\|T\|_e=0.$ So $[T_1]=[T_2]$ if and only if $T_1-T_2=i\partial \overline{\partial}u.$ This is why we say that $H_e$ is a space of classes of currents. It is easy  to show that $H_e$ is a Hilbert space. The important property is that it is possible to define on $H_e$ an intersection form.

\medskip

If $T_1,T_2$ are $\partial \overline{\partial}$ closed and smooth we can define the quadratic form
$$
Q(T_1,T_2)=\int T_1 \wedge T_2 \wedge \omega^{k-2}.$$

Using the decomposition for $T_j$, 

$$
T_j=c_j\omega+\partial S_j+\overline{\partial S_j}+i\partial \overline{\partial}u_j$$

\noindent we obtain by integration by parts that

$$Q(T_1,T_2)=\int T_1 \wedge \omega^{k-1}\int T_2\wedge \omega^{k-1}-2{\mbox{Re}}\int \overline{\partial}S_1
\wedge \partial \overline{S}_2 \wedge \omega^{k-2}.\;\;(1)
$$

The above expression makes sense for $T\in H_e$ and hence defines a quadratic form $Q(T_1,T_2)=\int T_1 \wedge T_2 \wedge \omega^{k-2}$ on $H_e,$ which is continuous for the $H_e$ topology. Indeed

$$
|Q(T_1,T_2)|\leq C\|T_1\|_e\|T'_2\|_e.
$$ 
The form $Q(T,T)$ is also clearly upper semi continuous for the weak topology of $H_e.$

\begin{proposition}
If a sequence $(T_n)$ of positive $\partial \overline{\partial}$ closed currents converge to $T,$ then $[T_n]$ converge to $[T]$ weakly in $H_e.$ Hence $T \rightarrow Q(T,T)$ is u.s.c. for the weak topology of currents.
\end{proposition}

\begin{proof}
Consider the decomposition of $T_n=c_n \omega+\partial S_n+
\overline{\partial S}_n+i\partial \overline{\partial}u_n.$ Since
$(T_n)$ converge weakly to $T,$ we can assume that all the $c_n=1$
and we have
$$
2\int \overline{\partial}S_n \wedge \partial \overline{S}_n \wedge \omega^{k-2}\leq 1.
$$

The $(S_n)$ are constructed using canonical solutions of the equation
$\overline{\partial}T_n=\partial \overline{\partial}S_n.$ So the weak limits of
$
\overline{\partial}S_n$ coincides with $\overline{\partial}S.$ Hence
$[T_n]\rightarrow [T]$ weakly in $H_e.$ Since $T \rightarrow Q(T,T)$ is u.s.c. with respect to the weak topology in $H_e$ it is also u.s.c. with respect to the weak topology on positive $\partial \overline{\partial}$ closed currents.

\end{proof}

The intersection form $Q$ is useful because of the following Hodge-Riemann property.

\begin{theorem}
The quadratic form $Q$ is negative definite on the hyperplane defined by
$$
{\mathcal H}=\{T;[T]\in H_e,\int T \wedge \omega^{k-1}=0\}.
$$
In particular if $[T]$, $[T']$ are two non proportional classes of positive $\partial \overline{\partial}$ closed currents, then $Q(T,T')>0.$ The function $Q$ is strictly concave on the cone of classes of positive $\partial \overline{\partial}$ closed currents.
\end{theorem}

\begin{proof}
We prove the first two claims. Since
$$
Q(T,T)=\left|\int T \wedge \omega^{k-1}\right|^2-2\int \overline{\partial}S
\wedge \partial \overline{S}\wedge \omega^{k-2},
$$
\noindent it is clear that for $T\in \mathcal H$, $Q(T,T)\leq 0.$ If
$Q(T,T)=0$ then $\overline{\partial}S=0$ hence $T$ is closed. Since $<T,\omega^{k-1}>=0$ then $T=i\partial \overline{\partial}u$ so $[T]=0.$

\medskip

Suppose that the space generated by $[T],[T']$ is of dimension
$2.$ Then there is an $a>0$ such that $T-aT'\in {\mathcal H}.$ Hence
$Q(T'-aT,T'-aT)<0.$ Expanding we get
$$
a^2Q(T,T)-2aQ(T,T')+Q(T',T')<0.$$ Since $Q(T,T)\geq 0$
and $Q(T',T')\geq 0$ and $a>0$, we get that $Q(T',T)>0.$

\end{proof}

We consider the following problem. Assume $(X,\mathcal L,E)$ is a lamination with singularities in $\mathbb P^2$, without positive closed directed currents. We want to find conditions on the singularity set which imply that there is a $\underline{\mbox{unique}}$ positive current of mass one, harmonic and directed by $\mathcal L.$ 
A strategy is to show that if $T_1,T_2$ are two such currents then $Q(T_1,T_2)=0.$  However, this implies only that the classes $[T_1],[T_2]$ of the currents in $H_e$ are equal. We then need the following result.

\begin{proposition}
Let $T_1,T_2$ be positive harmonic currents directed by $(X,\mathcal L,E)$
on a lamination in $\mathbb P^2$ without positive directed closed currents.
Assume $E$ is complete pluripolar and that $[T_1]=[T_2].$ Then $T_1=T_2.$
\end{proposition}

\begin{proof}
Let $B$ be a flow box away from $E.$ Then
$$
T_j=\int h^\alpha_j[V_\alpha]d\mu_j(\alpha), j=1,2.
$$
Let $\nu(\alpha)=\mu_1(\alpha)+\mu_2(\alpha),$ so $\mu_j=r_j(\alpha)\nu.$
Then
$$
T_1-T_2= \int (h_1^\alpha r_1(\alpha)-h_2^\alpha r_2(\alpha))[V_\alpha] d\nu(\alpha).$$
Since $[T_1]=[T_2]$, then $T_1-T_2$ is closed and hence
$h_1^\alpha r_1(\alpha)-h_2^\alpha r_2(\alpha)=c(\alpha)$ is constant.
We decompose $c(\alpha)\nu(\alpha)$ on the space of plaques, $c(\alpha)\nu(\alpha)=\lambda_1-\lambda_2$ for positive mutually singular positive measures. Then
$$
T_1-T_2=\int [V_\alpha]\lambda_1(\alpha)-\int [V_\alpha]\lambda_2(\alpha)
=T^+-T^-$$
\noindent for positive closed currents $T^\pm$. These currents fit together to  global positive closed currents on $X \setminus E.$ The mass of $T^\pm$ is bounded near $E.$ So since $E$ is complete pluripolar (or if $\Lambda^1(E)=0$)
the currents $T^\pm$ extend as closed currents through $E$, see for example \cite{De 1997} . Consequently
since $(X,\mathcal L)$ does not admit directed closed currents,
then $T^\pm=0$ and $T_1=T_2$. Indeed both currents have no mass on $E$, which satisfies $\Lambda^2(E)=0.$

\end{proof}

Once uniqueness is proved we get the following result, which can be considered as an ergodic theorem 
for appropriate averages on leaves.
\begin{proposition}
Assume that there is only one positive $( X,\mathcal L,E)$  directed harmonic current $T$ of mass $1$ on $X\subset \mathbb P^2.$ Then for every leaf $L$ for which a covering map $\phi:\Delta \rightarrow L$ satisfies the estimate $\int |\phi'(\zeta)|^2 (1-|\zeta|)=\infty,$
$$
\tau_r=\frac{\phi_*(G_r[\Delta_r])}{\|\phi_*(G_r[\Delta_r])\|}
\rightarrow T$$
\noindent when $r \rightarrow 1$ where $G_r=\log^+\frac{r}{|\zeta|}.$
\end{proposition}

\begin{proof}
We have seen in Theorem 22 that cluster points for $\tau_r$ are harmonic and  directed by $\mathcal L$, hence they coincide with $T.$
\end{proof}

We will next show uniqueness results for harmonic currents. With these results, discussed below we obtain the following corollary,
\cite{FS2005}, \cite{FS2006}.

\begin{corollary}
If $X$ is laminated without singularities and transversally Lipschitz or if
$(\mathbb P^2,\mathcal L,E)$ is a holomorphic foliation of $\mathbb P^2$ with only hyperbolic singularities, then $\tau_r$ converges to $T.$
\end{corollary}
Observe that in the cases described in the corollary, the currents $$
\tau_r=\frac{\phi_*(G_r[\Delta_r])}{\|\phi_*(G_r[\Delta_r])\|}$$ have directed limits. Indeed  the parametrized leaves are locally contained in leaves.

\subsection{Uniqueness of harmonic currents $T$ for $\mathbb P^2$}

\begin{theorem}
Let $\mathcal F$ be a holomorphic foliation in $\mathbb P^2$. Assume that all singular points are hyperbolic and there are no algebraic leaves. Then there is a unique positive harmonic $(1,1)$ current $T$ of mass one directed by ${\mathcal F}.$
\end{theorem}

According to Jouanolou's theorem (Theorem 7 in $\mathbb P^2$), for every $d \geq 2$, the hypothesis is satisfied in an open real Zariski dense set in $\mathcal F_d.$

\begin{theorem} Let $(X,\mathcal L)$ be a laminated set in $\mathbb P^2$, transversally Lipschitz without singularities. Then $X$ admits a unique positive $(1,1)$ harmonic current of mass one, directed by  
${\mathcal L}.$
\end{theorem}

\begin{remark}
Recently Deroin-Klepsyn \cite{DK2006} have obtained a different proof of this result, assuming that $X$ is minimal for a transversally conformal foliation. Their  approach uses diffusion for the heat equation on leaves as developed by Garnett \cite{G1983} and Candel \cite{C2003}.
Deroin-Klepsyn obtain also results on uniqueness of harmonic measure on minimal sets for general transversally conformal foliations, without singularities.
\end{remark}
To get uniqueness of directed harmonic currents we have to prove that for any laminated positive harmonic $(1,1)$ current $T, Q(T,T)=0.$ For that purpose we introduce the notion of geometric intersection and we compare it with the cohomological intersection given by the form Q.

\subsection{Geometric intersection of directed harmonic currents}

We will denote by $\Phi^\epsilon$ an automorphism of $\mathbb P^2$ close to the identity and set $T_\epsilon:=\Phi^\epsilon_*(T),\mathcal F_\epsilon=\Phi^\epsilon_*\mathcal F.$ The geometric intersection
of $T\wedge_g T_\epsilon$ is a measure which has the following expression in a flow box $B.$
Suppose $T=\int h_\alpha [\Delta_\alpha]d\mu(\alpha)$ and
$T_\epsilon=\int h^\epsilon_\beta[\Delta^\epsilon_\beta]d\mu(\beta).$ For a continuous compactly supported function $\phi$ in $B$ we define

$$
<T\wedge_gT_\epsilon,\phi>=
\int \left(\sum_{p\in J^\epsilon_{\alpha,\beta}}h_\alpha(p)h^\epsilon_\beta(p)\phi(p)\right) d\mu(\alpha)d\mu(\beta),$$

\noindent where $J^\epsilon_{\alpha,\beta}$ consists of the intersection points of $\Delta_\alpha$ and $\Delta^\epsilon_\beta.$ 

The automorphism is generic in order that the foliation $\mathcal F$ and
$\mathcal F_\epsilon$ have no common leaf. It also satisfies other generic conditions like moving all singular points away from the separatrices.

The delicate point is to show that $T\wedge_gT_\epsilon$ converges to zero and hence the limit has total mass zero. It is also necessary to show that the total mass is $\int T \wedge T$ which implies the uniqueness of $T$. We give a few details.

Define $S^\delta=\int \Phi_*(S) d \rho_\delta(\Phi)$ where $\rho_\delta$ is a smooth approximation to the identity in $U(3).$ The important point is that the estimate of the mass of $T \wedge_g T_\epsilon$ are independent of small perturbations of $\Phi^\epsilon.$

\begin{proposition}
With the above notations there are $\delta,\delta'$ small enough with respect to $\epsilon$ so that
$$
\int T^{\delta} \wedge_g T^{\delta'}_\epsilon \rightarrow \int T \wedge T..
$$
\end{proposition}

So the main problem is to estimate $T \wedge_g T_\epsilon.$

We first explain the idea in the nonsingular case. In that case it can be shown that the number of points of the intersection of the plaques $\Delta_\alpha$ and $\Delta^\epsilon_\beta$ is  bounded above independently of $\epsilon.$  The measure $\mu$ has no Dirac mass (otherwise there would be a closed current). Hence in any flow box the mass of $T^{\delta} \wedge_g T^{\delta'}_\epsilon$ converges to zero. To prove the boundedness of the number of intersections one has to control the intersections in far away flow boxes. The idea is that if this number is too large then using that the lamination is Lipschitz and following the flow boxes, the leaves in a fixed flow box would be closer than permitted from the explicit form of the automorphisms.

In the presence of singularities the above analysis holds away from the singularities. Using the Poincar\'e linearization theorem it is then enough to understand the repartition of mass of a harmonic current near a singular point, say $(0,0).$

Suppose $\mathcal F$ is given by $zdw-\lambda wdz=0, \lambda=a+ib, b>0.$
The leaves $L_\alpha$ in the unit polydisc are parametrized by $(z,w)=\psi_\alpha(\zeta)$ with

\bea
z & = & e^{i(\zeta+(\log |\alpha|)/b)}, \zeta=u+iv\\
w & = & \alpha e^{i \lambda(\zeta+(\log|\alpha|)/b)}\\
|z| & = & e^{-v}\\                   
|w| & = & e^{-bu-av}\\
\eea

The parametrization is chosen so that the plaques in the unit polydisc are parametrized by a fixed sector $S_\lambda$ defined by $ v>0, u>\frac{-av}{b}.$

Let $h_\alpha$ denote the harmonic function associated to the current $T$ on the leaf $L_\alpha.$ The local leaf clusters on both separatrices. To investigate the clustering in the $z-$ axis, we use a transversal $D_{z_0}:=\{(z_0,w);|w|<1$ for some $|z_0|=1.$ We can normalize so that $h_\alpha(z_0,w)=1$ where $(z_0,w)$ is the point on the local leaf with $e^{-2\pi b}\leq |w|<1.$ So $(z_0,w)=\psi_\alpha(\zeta_0)=\psi_\alpha(u_0+iv_0)$ with $v_0=0$ and $0<u_0 \leq 2\pi$ determined by the equations $|z_0|=e^{v_0}=1$ and $e^{-2\pi b}\leq |w|=e^{-bu_0-av_0}<1.$ Let $\tilde{h}_\alpha$ denote the harmonic continuation along $L_\alpha.$ Define $H_\alpha(\zeta):=\tilde{h}_\alpha(e^{i(\zeta+(\log |\alpha|)/b},\alpha e^{i\lambda(\zeta+(\log |\alpha|)/b)})$ on $S_\lambda.$

The nature of the holonomy map associated to the singularity permits to show that $H_\alpha$ is given by a Poisson integral with precise estimates on the behaviour at infinity, i.e. when the leaves get close to the separatrices. The main ingredient is the finiteness of the total mass of the harmonic currents on the disjoint flow boxes crossed when we follow a path around a hyperbolic singularity.We need also precise estimates on the location of the intersection points of the plaques in $\Delta^2.$ This gives that

$$
\int\sum_{p\in J^\epsilon_{\alpha,\beta}} H_\alpha(p)H^\epsilon_\beta(p)
d\mu(\alpha) d\mu(\beta) \rightarrow 0.
$$

 This shows that $\int T \wedge T=0$ for every $T$, positive harmonic, directed by $\mathcal F$. Hence using Theorem 25 and Proposition 6,$T$ is unique.The estimates are very technical. 

\begin{corollary}
Let $(X,\mathcal L,E)$ be a laminated set, with a unique directed harmonic current of mass one.Then any Borel set of leaves of $T$ of positive mass is of full $T$- measure.\end{corollary}

\begin{proof}
This follows from the uniqueness of $T.$ If $A$ is a Borel invariant set then the current $1_AT$ is positive harmonic. If it is not zero, then it should be proportional to $T.$.\end{proof}

\subsection{Minimal exceptional sets.}

Let $\mathcal F\in \mathcal F_d(\mathbb P^k).$ Recall that a closed minimal invariant set $M$ is a minimal exceptional set if $M \cap$Sing$(\mathcal F)=\emptyset.$ Whether such sets exist is a central open question for holomorphic foliations in $\mathbb P^2 $.

\begin{proposition}
Let $\mathcal F$ be a holomorphic foliation in $\mathbb P^2 $. The following are equivalent:

\noindent 1) $\mathcal F$ admits no minimal exceptional set.\\
2) The support of every positive harmonic or closed current directed by $\mathcal F$ intersects  Sing$(\mathcal F).$\\
3) There exists a positive harmonic or closed current directed by $\mathcal F,$ whose support intersects Sing$(\mathcal F).$
\end{proposition}

\begin{proof}
$1) \Rightarrow 2):$ Indeed the support of any directed current (closed or harmonic) is invariant under $\mathcal F$, so it has to intersect Sing$(\mathcal F).$ 

$2)\Rightarrow 3)$ is clear. We show that $3) \Rightarrow 1).$ Suppose that $\overline{L}\cap $Sing$(\mathcal F)=\emptyset.$ Let $T$ be a harmonic current directed by $\mathcal F$, with Supp$(T) \cap$Sing$(\mathcal F)\neq \emptyset.$ Let $S$ be a harmonic current directed by $\mathcal F$, 
supported by $\overline{L}.$ Then $S \neq T.$ On the other hand the geometric intersection $T\wedge_gS=0.$ So $\int T \wedge S=0$ and hence $T$ is proportional to $S,$ a contradiction. If $T$ or $S$ or both are closed the argument is similar.
\end{proof}

\begin{corollary}
If $\mathcal F$ admits a positive closed directed
 current, then there is no minimal exceptional set for $\mathcal F.$
\end{corollary}

\begin{proof}
We have seen in Theorem 17, that if $T$ is closed directed by $\mathcal F$, then \\
Supp$(T)\cap$Sing$(\mathcal F)\neq \emptyset.$ So we can apply the previous proposition.
\end{proof}
Let $\mathcal P_d$ denote the space of foliations in $\mathcal F_d(\mathbb P^2)$ with all singularities
in the Poincar\'e domain. We denote by $\mathcal P_d^0$ the subset of $\mathcal P_d$ such that $\mathcal F \in \mathcal P_d^0$ if and only if $\mathcal F$ does not admit a minimal exceptional set, i.e. 
every leaf $L$ for $\mathcal F$ clusters at a singular point.

\begin{proposition}
The set $\mathcal P_d^0$ is open in $\mathcal P_d.$ Moreover $\mathcal P_d^0$ contains the foliations in $\mathcal P_d$ admitting a positive closed directed current.
\end{proposition}

\begin{proof}
The second assertion is a consequence of Proposition 9. Observe that it is easy to construct such foliations for some degrees: Take a linear foliation $\mathcal L$ with all singularities in the Poincar\'e domain and consider $f^*(\mathcal L)$ with $f$ a generic holomorphic endomorphism such that the singularities are still in the Poincar\'e domain. The foliation $f^*(\mathcal L)$ is defined by the pull-back of the forms defining 
$\mathcal L$.

Fix $\mathcal F_0\in \mathcal P_d^0$ with singular points $(p_i)_{i\leq N}.$ There is a neighborhood $\mathcal U(\mathcal F_0)$ of $\mathcal F_0$ in $\mathcal F_d(\mathbb P^2)$ and a $\delta>0$ such that for every $\mathcal F \in \mathcal U(\mathcal F_0)$ the domain of linearization of $\mathcal F$ at $p_i(\mathcal F)$ contains $B(p_i,\delta).$ This is clear since linearization depends continuously on parameters in the Poincar\'e domain, as it is obtained using a fixed point Theorem.

We show that for $\mathcal F \in \mathcal U(\mathcal F_0)$ there is a harmonic current $T_\mathcal F$, directed by $\mathcal F,$ with mass on $\cup_{i=1}^N B(p_i,\delta)$ and hence containing in its support singular points of $\mathcal F.$ If not, let $\mathcal F_n \rightarrow \mathcal F_0$ and $T_{\mathcal F_n}$ with no mass on $\cup_{i=1}^N B(p_i,\delta).$ Let $T_0$ be a cluster point. Then $T_0$ has no mass on $\cup_{i=1}^N B(p_i,\delta),$ contradicting that  $\mathcal F_0\in \mathcal P^0,$ we use Proposition 9.

\end{proof}

\begin{problem}
Show $\mathcal P^0_d$ is closed in $\mathcal P_d.$
\end{problem}
\begin{proposition}
Let $\mathcal F$ be a holomorphic foliation in $\mathbb P^2$. Assume that all singular points are hyperbolic and there are no algebraic leaves. Every leave is dense in $\mathbb P^2$ iff the support of the harmonic current directed by F is equal to $\mathbb P^2$ 
\end{proposition}
\begin{proof}
Indeed the closure of any leaf contains the support of a directed harmonic current, since the harmonic current is unique it should contain $\mathbb P^2$ .The converse is clear.

\end{proof}

\subsection{Limit set}

Let $L$ be a leaf in a lamination with singularity $(X,\mathcal L, E)$. We define the limit set of $L$ as
$$
L_\infty=\cap_{n \geq 1} \overline{L \setminus K_n}
$$
\noindent where $(K_n)$ is an increasing sequence of compact sets in $L$ with $K_n \subset K_{n+1}^0$ and $\cup K_n=L.$ Clearly 
$L_\infty$ is independent of $(K_n)$.
The set $L_\infty\setminus E$ is invariant under $\mathcal L.$ We introduce also the set

$$
\mathcal L_\infty=\overline{\cup_{L \in \mathcal L}L_\infty}.
$$
So ${\mathcal L}_\infty$ is closed and invariant under $\mathcal L.$ It is of interest to describe $\mathcal L_\infty$ the limit set for ${\mathcal F}_d(\mathbb P^2).$ When $\mathcal L_\infty$  is small, then the situation is easily described.
\begin{theorem}
Let $(X,\mathcal L, E)$ be a lamination with singularities in $\mathbb P^2.$ Assume $\mathcal L_\infty$ has vanishing one dimensional Hausdorff measure. Then all leaves are closed. If $\mathcal L\in \mathcal F_d(\mathbb P^2)$ then $\mathcal L$ admits a meromorphic first integral.
\end{theorem}
\begin{proof}
Indeed by Harvey's Theorem \cite{H1974}  the closure of each leaf is analytic, so the leaves are closed. When $\mathcal L\in \mathcal F_d(\mathbb P^2)$ one can  apply Theorem 6.
\end{proof}
The case where $\mathcal L_\infty$ is of dimension one has been desribed by Camacho-Lins-Neto-Sad, \cite{CLS1988}.

\begin{theorem}
(Camacho-Lins-Neto-Sad) Let $\mathcal L\in \mathcal F_d(\mathbb P^2)$. Assume $\mathcal L_\infty$
is analytic of dimension $1,$ satisfying the following assumptions.\\
(i) The holonomy of each irreducible component of $\mathcal L_\infty$ is hyperbolic.\\
(ii) The number of separatrices at each singularity is finite.\\
Then there is a rational map $F$ of $\mathbb P^2$ and a linear flow $\mathcal L_0$ such that $\mathcal L=F^* \mathcal L_0.$
\end{theorem}

A consequence of the uniqueness of harmonic currents is the following.

\begin{theorem}
Let $\mathcal L \in \mathcal H(d)$, i.e. all singular points are hyperbolic and there is no algebraic leaf. Then for every leaf $L,$
$$
{\mbox{Supp}}(T) \subset L_\infty,
$$
\noindent where $T$ is the unique positive harmonic current of mass $1$ directed by $\mathcal L.$
\end{theorem}

Indeed $L_\infty$ is invariant, not finite, hence it supports a positive  harmonic non closed current which is necessarily equal to $T.$

\subsection{Harmonic Flow}

Let $(X,\mathcal L,E)$ be a lamination with singularities in a compact Hermitian manifold $M.$
We consider $\mathcal C_{\mathcal L}$, the convex  set of positive harmonic currents of mass $1$, directed by $\mathcal L.$ We show that $\mathcal C_{\mathcal L}$ is compact. Let $T_n\in \mathcal C_{\mathcal L}$.
Suppose $T_n \rightarrow T.$ We have to show that $T$ is directed. In a flow box $B$ we have $(T_n)_{|B}$
are in the closed convex hull of the directed harmonic currents $h_\alpha [V_\alpha] \delta_\alpha,
h_\alpha(0)=1,$ $\delta_\alpha$ the Dirac mass at $\alpha.$ So $T$ is directed in $B.$ and $\mathcal C_{\mathcal L}$ is compact. 
Each $T\in \mathcal C_{\mathcal L}$ is expressed in a flow box $B$ as
$$
T=\int h_\alpha [V_\alpha]d\mu(\alpha).
$$
We also have $\partial T=\tau\wedge T,$ with $\tau=\frac{\partial h_\alpha}{h_\alpha},$ here $\partial$ is the $\partial-$ operator along leaves. Observe that if we replace $h_\alpha$ by $c_\alpha h_\alpha$ and accordingly $\mu$ by $\frac{\mu(\alpha)}{c_\alpha}$, then $\tau$ is unchanged. By Harnack, $\tau$ is bounded in every flow box.

To simplify our description we assume that $T$ is extremal in $\mathcal C_{\mathcal L}$.

We define a metric $g_T$ along leaves: $g_T=\frac{i}{2} \tau \otimes \overline{\tau}.$ If we choose coordinates $(z,\alpha)$ in a flow box then
$$
g_T=\frac{i}{2} \left|\frac{\partial h_\alpha}{\partial z_\alpha}\right|^2 \frac{1}{h_\alpha^2}dz_\alpha \otimes d\overline{z}_\alpha.
$$
Observe that the metric is defined $T$ a.e. and that  it's a metric only out of the set $\mathcal C_T=\{(z,\alpha); \frac{\partial h_\alpha(z)}{\partial z_\alpha}=0\}.$

We can also introduce the positive measure

$$\mu_T:=i \tau \wedge \overline{\tau}\wedge T.
$$

The measure is locally finite out of $E.$ See Frankel \cite{Fr1995} for the non singular case.

\begin{proposition}
Let $T$ be a non closed extremal current in $\mathcal C_{\mathcal L}.$ The metric has constant negative curvature out of $\mathcal C_T$, which is of $T$ measure zero.
\end{proposition}

\begin{proof}
Consider $\mathcal N_g:=\{{\mbox{leaves on which}}\; g_T\; {\mbox{vanishes identically}}\}$.The set $\mathcal N_g$ is measurable and if it is of $T$ positive measure, the extremality of $T,$ would imply that it is of full measure in which  case $T$ is closed, a contradiction. So we can assume that the metric $g_T$ is non zero on a set of full $T$ measure.
Since $h_\alpha$ is harmonic along leaves, $\partial \overline{\partial} (\log h_\alpha)=-\frac{\partial h_\alpha}{h_\alpha}\wedge \frac{\overline{\partial} h_\alpha}{h_\alpha}=-\tau \wedge \overline{\tau}.$ 

The expression of the curvature of a conformal metric g is :

$$
     \kappa(g)=-\frac{1}{4} \frac{\Delta \log g}{g}=\frac{1}{2} \frac{\Delta \log h_\alpha}{\left|\frac{\partial h_\alpha}{\partial z_\alpha}\right|^2 \frac{1}{h_\alpha^2}}.
$$

So 

$$
\kappa(g_T)=  \frac{h_\alpha^2}{|h_{\alpha,z}|^2} \left( \frac{\partial}{\partial \overline{z}}
\left( \frac{h_{\alpha,z}}{h_\alpha}\right)\right).
$$

Since $h_\alpha$ is harmonic we get $\kappa(g_T)=-2.$

\end{proof}

We define the harmonic flow associated to $T.$ Consider $\chi_\alpha$, the gradient vector field associated to the ''function'' $h_\alpha.$ We compute the gradient with respect to the metric $g_T.$ So in a flow box with local coordinates $z_\alpha=x_\alpha+iy_\alpha$ we have
$\chi_\alpha=
c(h_{x_\alpha}, h_{y_\alpha})$ 
\noindent with $c$ chosen so that $g_T(\chi_\alpha,\chi_\alpha)=1.$  The vector field is independent of the choice of $h_\alpha$ so it is well defined $T$ a.e.
 The points in $\mathcal C_T$  are singular for the gradient of $h_\alpha.$ With the above normalization the parametrized integral curves for $
\chi_\alpha$ approach these points at infinite speed.

 The volume $\mu_T$ is finite when $E=\emptyset$ in any dimension and also in the following case.
 
 \begin{theorem}\cite{FS2006}  Let $\mathcal F \in \mathcal H(d),$ with $T$ the unique harmonic current associated to $\mathcal F.$ Then the measure $\mu_T$ is $\underline{\mbox{finite}}$ 
  \end{theorem}
 There is indeed a more precise result.
 If $(\mathcal F_\lambda)$ is a holomorphic family in $\mathcal H(d),\lambda \in \Delta(0,r),$ then the mass of $\mu_{T_\lambda}$ is, locally in $\lambda,$ uniformly small in a neighborhood of the singularities. The  proof uses heavily the contraction of holonomy and estimates of Poisson integrals.

\noindent John Erik Forn\ae ss\\
Mathematics Department\\
The University of Michigan\\
East Hall, Ann Arbor, MI 48109\\
USA\\
fornaess@umich.edu\\

\noindent Nessim Sibony\\
CNRS UMR8628\\
Mathematics Department\\
Universit\'e Paris-Sud\\
Batiment 425\\
Orsay Cedex\\
France\\
nessim.sibony@math.u-psud.fr\\
\end{document}